%% file: maxwell_siam.tex
\documentclass[onefignum,onetabnum
]{siamart171218}

\input{ex_shared}

\ifpdf
\hypersetup{
  pdftitle={Isogeometric Boundary Elements in Electromagnetism},
  pdfauthor={J. Doelz, S. Kurz, S. Schoeps, F. Wolf}
}
\fi

\begin{document}

\maketitle

\begin{abstract}
	We analyze a new approach to three-dimensional electromagnetic scattering problems via fast
	isogeometric boundary element methods. Starting with an investigation of
	the theoretical setting around the electric field integral equation within the
	isogeometric framework, we show existence, uniqueness, and quasi-optimality
	of the isogeometric approach. For a fast and efficient computation, 
	we then introduce and analyze an interpolation-based fast multipole method
	tailored to the isogeometric setting, which admits competitive algorithmic
	and complexity properties. This is followed by a
	series of numerical examples of industrial
	scope, together with a detailed presentation and interpretation of the 
	results.
\end{abstract}



\section{Introduction}
\input{sec_introduction}
\section{Fundamental Concepts}\label{sec::fundamentalconcepts}
\input{sec_basics}
\section{The Discrete Problem}\label{sec::discrete}
\input{sec_iga}
\section{Details of Implementation}\label{sec::implement}
\input{sec_disc}
\section{Numerical Examples}\label{sec::num}
\input{sec_num}
\section{Conclusion}\label{sec::conclusion}
\input{sec_conclusion}

\begin{footnotesize}
\section*{Acknowledgments}
The authors would like to thank Lucy Weggler for providing the numerical implementation of
the reference solution for the Mie scattering.
This work is supported by DFG Grants \emph{SCHO1562/3-1} and \emph{KU1553/4-1}
within the project \emph{Simulation of superconducting cavities with isogeometric
boundary elements (IGA-BEM)}. Jürgen Dölz is an \emph{Early Postdoc.Mobility
fellow}, funded by the Swiss National Science Foundation through the project
\emph{174987 H-Matrix Techniques and Uncertainty Quantification in Electromagnetism},
the \emph{Excellence Initiative} of the German Federal and State Governments
and the \emph{Graduate School of Computational Engineering} at TU Darmstadt. The
work of Felix Wolf is supported by the \emph{Excellence Initiative} of the German
Federal and State Governments and the \emph{Graduate School of Computational
Engineering} at TU Darmstadt.
\end{footnotesize}

\bibliographystyle{siamplain}
\input{sec_references}

\appendix
\section{Inverse Estimate}
\input{sec_appendix}

\end{document}

%% file: ex_shared.tex
\usepackage{stmaryrd}
		\SetSymbolFont{stmry}{bold}{U}{stmry}{m}{n}
\usepackage{placeins}
 
\usepackage[utf8]{inputenc}
\usepackage[english]{babel}
\usepackage{mathptmx}
\usepackage{lmodern}
\usepackage[T1]{fontenc}
\usepackage[final]{microtype}
\usepackage[mathscr]{euscript} 

\usepackage{mathtools}
\usepackage{stmaryrd}
\usepackage{subcaption}
\usepackage{comment}
\SetSymbolFont{stmry}{bold}{U}{stmry}{m}{n}

\usepackage[colorinlistoftodos,prependcaption,textsize=small]{todonotes}
\usepackage{amssymb}

\usepackage{pgf}
\usepackage{pgfplots}
\usepackage{tikz}
\usepackage{tikz-cd}
\usetikzlibrary{arrows, matrix, patterns,shapes,calc}

\usepackage{caption}

\newtheorem*{problem}{Problem}

\newcommand{\bb}{\pmb}
\newcommand{\dd}{\,\mathrm{d}}
\newcommand{\Id}{\operatorname{Id}}
\newcommand{\diam}{\operatorname{diam}}
\newcommand{\dist}{\operatorname{dist}}

\newcommand{\R}{{\mathbb R}}

\newcommand{\norm}[1]{{\left\lVert #1 \right\rVert}}
\newcommand{\abs}[1]{{\left\lvert #1 \right\rvert}}

\newcommand{\opd}{\,\operatorname{d}}
\renewcommand{\div}{\operatorname{div}}
\DeclareMathOperator{\curl}{{curl}}

\renewcommand{\S}{{\mathbb S}}

\renewcommand{\hat}{\widehat}
\DeclareMathOperator{\bcurl}{{\mathbf{curl}}}
\DeclareMathOperator{\grad}{grad}

\newcommand{\loc}{{\mathrm{loc}}}

\newcommand{\MSL}{{\bb{\mathscr  V}}}
\newcommand{\chispace}{ {\bb H^{-1/2}_\times(\div_\Gamma,\Gamma)}{}}

\newcommand{\Garding}{{G\r arding}}

\newcommand{\Dipole}{\mathbf{DP}}
\usepackage{adjustbox}

\usepackage{amsfonts}
\usepackage{graphicx}
\usepackage{epstopdf}
\usepackage{algorithm}
\usepackage{algpseudocode}
\ifpdf
  \DeclareGraphicsExtensions{.eps,.pdf,.png,.jpg}
\else
  \DeclareGraphicsExtensions{.eps}
\fi

\newsiamremark{remark}{Remark}
\newsiamremark{hypothesis}{Hypothesis}
\crefname{hypothesis}{Hypothesis}{Hypotheses}
\newsiamthm{claim}{Claim}

\headers{Isogeometric Boundary Elements in Electromagnetism}{J.~Dölz, S.~Kurz, S.~Schöps, and F.~Wolf}

\title{Isogeometric Boundary Elements in Electromagnetism: Rigorous Analysis, Fast Methods, and Examples
\thanks{Submitted to arXiv\today.
}}

\author{J\"urgen D\"olz\thanks{Technische Universit\"at Darmstadt,
Institute for Accelerator Science and Electromagnetic Fields, Centre for Computational Engineering, \email{[doelz/kurz/schoeps/wolf]@gsc.tu-darmstadt.de}.}
\and Stefan Kurz\footnotemark[2]
\and Sebastian Schöps\footnotemark[2]
\and Felix Wolf\footnotemark[2]
\thanks{Corresponding author.}
}



%% file: sec_introduction.tex

As has been shown since their introduction by \cite{Hughes_2005aa} in 2005, isogeometric methods offer a variety of advantages over triangulation-based approaches. Not only do they allow for an exact geometry representation via parametric mappings, but also improved spectral properties and higher accuracies per degree of freedom (DOF) have been shown in~\cite{Cottrell_2009aa}.

The interest in isogeometric boundary element methods is rooted in the need for volumetric mappings by classical isogeometric analysis (IGA). These are usually not provided by computer-aided design (CAD) frameworks since most CAD systems handle only boundary representations. 
Boundary element methods, which already exploited the possibility of exact geometry representations via parametric mappings before the introduction of IGA \cite{Harbrecht_2001aa}, avoid this problem by reducing the problem at hand to an integral equation operating exclusively on the boundary of the domain of interest. 
Moreover, since merely a boundary representation must be available, this makes boundary element methods exceptionally well suited for exterior problems, specifically scattering problems.
The adaptation of this to the isogeometric framework has spawned a whole area of research, see, e.g, \cite{Aimi_18aa,Dolz_2016aa,Dolz_2018aa,Feischl_2017aa,Harbrecht_2013ab,Marussig_2015aa, Simpson_2012aa,Takahashi_2012aa}, and the
references therein, {as well as alternative approaches aiming to omit discretization errors, e.g.~\cite{Lie_2016_aa}.}

Although the independence of volumetric mappings is one major advantage of isogeometric boundary element methods, there are obstacles to be aware of. First, they rely on the existence of fundamental solutions, thus not all PDEs are solvable by a boundary element approach. 
Second, due to the non-local formulation, the arising systems are dense, and thus boundary element methods
rely on so-called \emph{fast methods} to be used efficiently, see eg.~\cite{Beb00, greengard1987fast, HN89, Hackbusch_2002aa,Harbrecht_2001aa,Harbrecht_2013ab}.

While solid implementations and analysis for acoustic scattering exist, cf., e.g., \cite{Dolz_2018aa}, the efficient solution of the more involved electromagnetic problems is still an open area of research, but of vast interest for the electromagnetic engineering community \cite{Bontinck_2017ag}.
Implementations of (lowest-order) boundary element schemes for electromagnetic scatterings, mostly based on the \emph{electric field integral equation (EFIE)}, are widely adopted in industry, and classical implementations are well understood, cf.~\cite{Peterson_2006aa}. 
Implementations realizing the isogeometric approach as first presented by \cite{Buffa_2014aa} exist as well, cf.~\cite{Simpson_2018aa}.
However, many open questions about the theory and behavior of this approach remain open, and the potential of fast methods is far from being fully exploited.

\tikzstyle{decision} = [diamond, draw,  left color=blue!20, right color=green!20, text width=4.5em, text badly centered, node distance=3cm, inner sep=0pt]
\tikzstyle{solver} = [rectangle, draw, fill=blue!20, 
    text width=7em, text centered, rounded corners, minimum height=4em]
    \tikzstyle{block} = [rectangle, draw, fill=green!20, 
    text width=7em, text centered, rounded corners, minimum height=4em]
\tikzstyle{line} = [draw, -latex']
\tikzstyle{solcloud} = [draw, ellipse,text centered,fill=blue!20,text width=7em,minimum height=4em]
\tikzstyle{ancloud} = [draw, ellipse,text centered,fill=green!20,text width=7em,minimum height=4em]

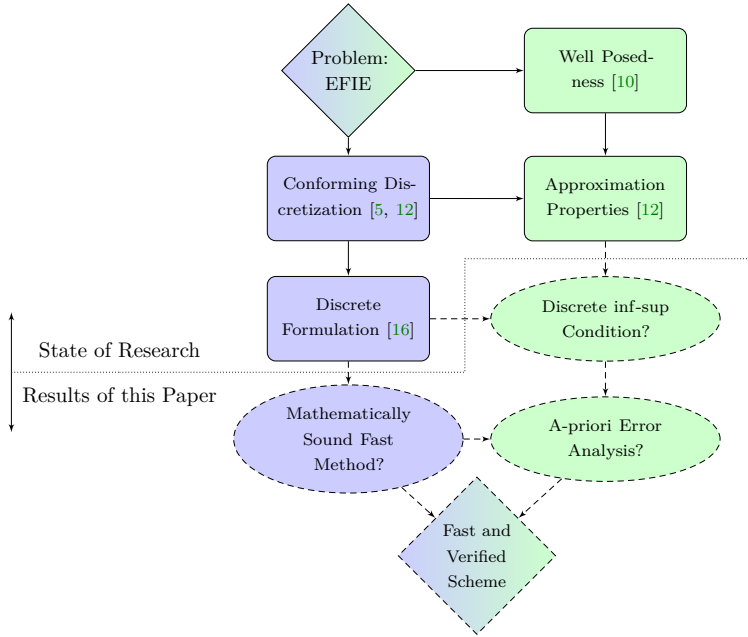
\begin{figure}
	\centering
\adjustbox{scale=.8,center}{%
    \begin{tikzpicture}[node distance = 2cm,auto]
    \node [decision] (model) {\small Problem: EFIE};
    \node [solver,below=.3cm of model] (disc) {\footnotesize Conforming Discretization~\cite{Beirao-da-Veiga_2014aa,Buffa_2018aa}};
    \node [solver,below of=disc] (form) {\footnotesize Discrete Formulation~\cite{Buffa_2014aa}};
    \node [solcloud,below of=form,densely dashed] (fastnum) {\footnotesize Mathematically Sound Fast Method?};

    \node [block, right=1.8cm of model] (contdef) {\footnotesize Well Posedness~\cite{Buffa_2003aa}};
    \node [block,below=.71cm of contdef] (approx) {\footnotesize Approximation Properties~\cite{Buffa_2018aa}};

    \node [ancloud,below of=approx,densely dashed] (infsup) {\footnotesize Discrete $\inf$-$\sup$ Condition?};
    \node [ancloud,below of=infsup,densely dashed] (error) {\footnotesize A-priori Error Analysis?};
    \node (ghost) at ($0.5*(error)+0.5*(fastnum)$) {\footnotesize };
    \node [decision,below=.5cm of ghost,densely dashed] (goal) {\footnotesize Fast and Verified Scheme};
    \path [line] (model) -- (contdef);
    \path [line] (disc) -- (approx);
    \path [line,densely dashed] (form) -- (infsup);
    \path [line,densely dashed] (fastnum) -- (error);
    \path [line] (model) -- (disc);
    \path [line] (disc) -- (form);
    \path [line,densely dashed] (form) -- (fastnum);
    \path [line] (contdef) -- (approx);
    \path [line,densely dashed] (approx) -- (infsup);
    \path [line,densely dashed] (infsup) -- (error);
    \path [line,densely dashed] (error) -- (goal); 
    \path [line,densely dashed] (fastnum) -- (goal);

    \node (lineleft) at ($.5*(approx)+.5*(infsup)+(2.5,0)$){};
    \node[label=above:{State of Research}] (lineright) at ($.5*(fastnum)+.5*(form)-(3.8,-.11)$) {};
    \node[label=below:{Results of this Paper}] at (lineright) {};
    \node (linemidbot) at ($(lineleft -| ghost)+(-.2,0)$) {};
    \node (linemidtop) at ($(lineright -| ghost)+(-.2,0)$) {};

    \path [draw,densely dotted] (lineleft.center) -- (linemidbot.center) -- (linemidtop.center) -- (lineright.center);
    \node (leftbound) at ($(lineright)+(-1.8,0)$){};
    \path [line] (leftbound.center) -- ($(leftbound)+(0,1)$);
    \path [line] (leftbound.center) -- ($(leftbound)+(0,-1)$);
    \path [draw,densely dotted] (leftbound.center) -- (lineright.center);
\end{tikzpicture}}
\caption{A diagram showcasing the intention of this paper. Blue represents crucial algorithmic advances, and green the corresponding analysis. It remarks on the state of research (blocks) of isogeometric formulations of the electric field integral equation (EFIE), as well as open questions, which have not yet fully been discussed in the literature (circled). We aim to discuss these within this paper and present a fast scheme which is verified by theory.
Note that for classical approaches to the electric field integral equation most of these issues have been resolved already.}\label{fig::purpose}
\end{figure}
The contribution of this paper is threefold, cf.~Figure~\ref{fig::purpose}.
First, existence and uniqueness of discrete solutions for the isogeometric
approach are unclear. Classical proofs for other discretizations rely on commuting interpolation operators, see \cite{Bespalov_2010aa,Buffa_2002aa,Buffa_2003ab},
which have been obtained only recently for the isogoemetric approach in a
multi-patch setting
\cite{Buffa_2018aa}. In view of these developments, we establishing a discrete $\inf$-$\sup$-condition
which yields existence and uniqueness for isogeometric discretizations of
the electric field integral equation. Together with recent approximation
results \cite{Buffa_2018aa}, this guarantees optimal convergence rates.

The second point this paper is concerned with is the need for an efficient fast method for this framework. The issue with usual methods, see, e.g., \cite{Beb00, greengard1987fast, HN89, Hackbusch_2002aa},
is the fact that they were designed for lower order trial spaces and iterate over the degrees of freedom -- rather than the elements --
during numerical quadrature. Applying this approach to higher order spline spaces as done in \cite{Simpson_2018aa}
results in difficult bookkeeping of the supports and expensive numerical quadrature,
since redundant evaluations of geometry and fundamental solution are inevitable.

In this paper, we follow the approach of \cite{Dolz_2016aa, WegglerMatrix}, which allows
for a fast method with element-wise quadrature and fully avoids redundant evaluations of
geometry and fundamental solution.
Our method exploits the isogeometric structure and yields a simplified
implementation based on interpolation on the unit square.
Moreover, our approach fits effortlessly into the
$\mathcal{H}^2$-matrix framework \cite{Borm_2010aa}, which is a more efficient
specialization of the frequently used $\mathcal{H}$-matrix framework \cite{Hac15}.
However, one should be aware that  approximations of the system by fast methods directly influences the quality
of the solution. Thus, induced errors needs to be well understood and controlled. 
We provide a detailed
analysis of the presented method which shows that it provably
maintains optimal convergence rates. To best of our
knowledge, this makes our method the only fast method for
isogeometric boundary elements in the case of the EFIE
which is mathematically sound and computationally efficient.

This document is structured as follows.
Section \ref{sec::fundamentalconcepts} reviews the basic notions required for an analysis of the electric field integral equation and its discretization within the isogeometric framework, which will be formulated in Section \ref{sec::discrete}. This is followed by a discussion of existence, uniqueness, and quasi-optimality of the solution to the arising variational problem.
Afterward, in Section \ref{sec::implement}, we introduce and analyze an interpolation-based fast multipole approach. 
Section \ref{sec::num} then introduces numerical experiments of different sizes, in which we investigate the behavior of our method, with both the surface current and the scattered field as a quantity of interest in mind. 
Finally, in Section \ref{sec::conclusion}, we conclude and reflect our results.

Throughout this paper, in order to avoid the repeated use of generic but unspecified constants, by $C \lesssim D$ we mean that $C$ can be bounded by a multiple of $D$, independently of parameters which $C$ and $D$ may depend on. In the usual sense, $C \gtrsim D$ is defined as $D \lesssim C$, and $C\simeq D$ as $C \lesssim D \lesssim C$.

\FloatBarrier

%% file: sec_basics.tex
We first introduce the required mathematical framework for a precise introduction of the scattering problem for which we follow the leads of \cite{Buffa_2003ab}. 
Afterward, we will introduce the electric field integral equation and review some of its properties.

\subsection{The Electromagnetic Scattering Problem}

On the bounded domain $\Omega\subset\mathbb{R}^3$ and for $0\leq s$, we denote by $H^s(\Omega)$ the usual Sobolev spaces \cite{McLean_2000aa}, and by $\bb H^s(\Omega)$ their vector valued counterparts. For $s=0$ we utilize the convention $H (\Omega) = H^0 (\Omega) = L^2(\Omega)$ and $\bb H (\Omega) = \bb H^0 (\Omega) = \bb L^2(\Omega)$. On unbounded domains $\Omega^c:=\mathbb{R}^3\setminus\overline{\Omega}$ we utilize the same notation together with the subscript ``$\loc$'' in the form of $H^s_\loc(\Omega^c)$ and $\bb H^s_\loc(\Omega^c)$ to denote that the required regularity conditions must only be fulfilled on all bounded subdomains of $\Omega^c$.

For compact manifolds $\Gamma$ we denote by $H^s(\Gamma)$ the usual construction of Sobolev spaces on manifolds via charts, and by $\bb H^s(\Gamma)$ their vector valued counterparts. As usual, we define the spaces $H^{-s}(\Gamma)$ and $\bb H^{-s}(\Gamma)$ as the dual spaces of $H^s(\Gamma)$ and $\bb H^s(\Gamma)$ w.r.t.~$L^2(\Gamma)$ and $\bb L^2(\Gamma)$ as pivot spaces.

Let $\mathcal{M}$ be one of the domains $\Omega$, $\Omega^c$, or the boundary $\Gamma$. For any differential operators $\operatorname{d}$ defined on $\mathcal{M}$, we define the spaces $H^s(\operatorname{d},\mathcal{M})$ via the closure of $H^s(\mathcal{M})$ under the graph norm $\norm{\cdot}_{H^s(\mathcal{M})}+\norm{\operatorname d(\cdot)}_{H^s(\mathcal{M})}$, equipping the spaces with the same. The definition of graph norms is generalised to vector-valued differential operators and spaces in complete analogy {and we denote by $\bb H^s(\div 0,\mathcal{M})$ the elements of $\bb H^s(\mathcal{M})$ with zero divergence.}
The following trace operator for vector fields onto Lipschitz boundaries will be required to describe meaningful boundary data to the electric wave equation.
\begin{definition}[Rotated Tangential Trace Operators, \cite{Buffa_2003ab}]
For $\bb u\in C(\Omega^c; \mathbb C^3)$, with $\Omega$ being a domain with Lipschitz boundary, we define the \emph{exterior rotated tangential trace operator} as
\begin{align*}
        \bb \gamma_{ t}^+ (\bb u)(\bb x_0) & \coloneqq \lim_{\substack{\bb x\to \bb x_0\\\bb x\in\Omega^c}}\bb u(\bb x) \times \bb{n}_{\bb x_0},\quad\text{for all}~\bb x_0\in\Gamma,
\end{align*}
where $\bb{n}_{\bb x_0}$ denotes the exterior normal vector of $\Omega$ at $\bb x_0$. {The interior trace $\bb \gamma_{t}^-$ is defined accordingly, using the exterior normal.}
\end{definition}
By density arguments, see also \cite{Buffa_2003ab}, this notation can be extended to be applicable to the spaces $\bb H^{s+1/2}_{\loc}(\Omega^c)$ for $0<s<1$ and $\bb H_\loc(\bcurl,\Omega^c)$. Thus, we define  $\bb H^s_\times(\Gamma)\coloneqq \bb \gamma_t^+\big(\bb H^{s+1/2}_{\loc}(\Omega^c)\big)$ for all $0<s<1$ as well as
\begin{align*}
	\chispace \coloneqq \bb \gamma_t^+\big(\bb H_\loc(\bcurl,\Omega^c)\big).
\end{align*}
It is known that $\bb \gamma_t^+ \colon\bb H_\loc(\bcurl,\Omega^c)\to\chispace$ is a bounded linear operator \cite{Buffa_2003ab}.

With respect to the pairing
\[
\langle \bb\mu,\bb\nu\rangle_{\times} = \int_\Gamma (\bb\mu\times\bb n_{\bb x}) \cdot \bb\nu \opd \sigma_{\bb x},
\]
we define the spaces $\bb H^{-s}_\times(\Gamma)$ by duality to $\bb H^s_\times(\Gamma)$ for $0<s<1$.
Note, however, that the space $\chispace$ cannot be defined via such a duality if $\Gamma$ is non-smooth, cf.~\cite{Buffa_2003ab}.

Given a perfectly conducting object $\Omega$ with Lipschitz boundary $\Gamma$ in a surrounding $\Omega^c$, we are interested in the scattered field $\bb e_s$ of an electric incident wave $\bb e_i$ hitting the scatterer $\Omega$. Assuming a time-harmonic problem, the scattered field $\bb e_s$ can then be described in the frequency domain by the \emph{electric wave equation}
\begin{align}\label{problem::ext_scattering}\left\lbrace\qquad
\begin{aligned}
	\bcurl \bcurl\, \bb e_s -\kappa^2\bb e_s &= 0&&\text{in}~\Omega^c,\\
	\bb\gamma_t^+\bb e_s &= -\bb\gamma_t^+\bb e_i&&\text{on}~\Gamma,\\
	\big|\bcurl\,\bb e_s(\bb x)\times{\bb x}\cdot{|\bb x|}^{-1}-i\omega\varepsilon _0\bb e_s(\bb x)\big|&=\mathcal{O}(|\bb x|^{-2}),&&|\bb x|\to\infty.
\end{aligned}\right.
\end{align}
The wavenumber $\kappa=\omega\sqrt{\varepsilon_0\mu_0}$ is described in terms of the frequency $\omega$, as well as the material parameters \emph{permittivity} $\varepsilon_0>0$ and \emph{permeability} $\mu_0>0$, which we assume to be constant.
It is known that \eqref{problem::ext_scattering} is uniquely solvable for any sufficiently regular Dirichlet data and wavenumbers $\kappa >0$, see \cite{Buffa_2004aa}.
Given an incident wave $\bb e_i$, the total electric field $\bb e$ in $\Omega^c$ is then given by $\bb e = \bb e_i+\bb e_s$.

\subsection{The Electric Field Integral Equation}
Since \eqref{problem::ext_scattering} is an unbounded exterior problem in a homogeneous medium, it is convenient to use the following boundary integral representation.
\begin{lemma}[Representation Formula, \cite{Buffa_2003ab}]
	For any solution $\bb e_s$ of \eqref{problem::ext_scattering} there exists a density $\bb w\in \chispace$ such that $\bb e_s (\bb x) = (\tilde{\MSL}_\kappa\bb w)(\bb x)$ for all $\bb x\in \Omega^c$, where
	\begin{align}\label{eq:ESLP}
		(\tilde{\MSL}_\kappa\bb w)(\bb x)
		=
		\int _\Gamma G_\kappa(\bb x-\bb y)\bb w(\bb y)\dd\sigma_{\bb y}
		+
		\frac{1}{\kappa^2}
		\bb\grad_{\bb x}
		\int _\Gamma G_\kappa(\bb x-\bb y)\div_{\Gamma}\bb w(\bb y)\dd\sigma_{\bb y}.
	\end{align}
	$G_\kappa(\bb x-\bb y)$ is herein given by the
	Helmholtz fundamental solution
	\begin{align}
	\label{eq:HelmholtzGreen}
		G_\kappa(\bb x-\bb y)=\frac{e^{i\kappa\|\bb x-\bb y\|}}{4\pi\|\bb x-\bb y\|}.
	\end{align}
	Moreover, the \emph{electric single layer potential} given in \eqref{eq:ESLP} is a continuous operator
		$\tilde{\MSL}_{\kappa}\colon \bb H^{-1/2}_\times(\div_\Gamma,\Gamma)\to\bb H_\loc(\bcurl\,\bcurl,\Omega^c),$
	such that the image of $\tilde{\MSL}_\kappa$ is divergence free within $\Omega^c$.
\end{lemma}
We remark that the very same representation formula holds also for the electric wave equation in the \emph{bounded} domain $\Omega$, which we shall not need here. However, for our following considerations it is important to keep in mind that the interior and the exterior problem are closely related to each other. More precisely, the following considerations for the \emph{exterior} problem fail, if $\kappa$ is a resonant wavenumber of the \emph{interior} problem, see \cite{Buffa_2003ab} for
a precise definition and discussion.

By the Lemma above we know that a density $\bb w$, {which has a physical meaning in terms of a surface current,} with
$
\bb e_s=\tilde{\MSL}_{\kappa}\bb w,
$
exists.
To obtain it, we apply the tangential trace on both sides of \eqref{eq:ESLP}, which yields the \emph{electric field integral equation}
\begin{align}\label{eq:EFIE}
	-\bb\gamma_t^+\bb e_i=(\bb\gamma_t^+\tilde{\MSL}_\kappa)(\bb w)=:\MSL_\kappa\bb w.
\end{align}

The variational formulation for the electric field integral equation \eqref{eq:EFIE} is as follows.
\begin{problem}[Continuous Problem]
	Find $\bb w\in \chispace$ such that
	\begin{align}
		\langle \MSL_{\kappa}\bb w,\bb\xi \rangle_\times = -\langle \bb\gamma_t^+\bb e_i,\bb\xi \rangle_\times, \label{problem::variational::cont} 
	\end{align}
	for all $\bb\xi \in \chispace$.
\end{problem}
As done in \cite{Buffa_2003aa}, one can utilize a generalized \Garding{}-inqequality to show well posedness of this continuous problem for non resonant wavenumbers. 

%% file: sec_iga.tex

We will now discuss the technical details and analytic properties for the discretization of \eqref{problem::variational::cont}.
Since much of the following analysis is based on the approximation results of \cite{Buffa_2018aa}, we will follow its notation closely. For a more in-depth introduction to spline theory we refer to \cite{Schumaker_2007aa}, and to its application to variational isogeometric analysis to \cite{Beirao-da-Veiga_2014aa}.

\subsection{Fundamental Notions}\label{sec::subsec::iga}
We review the basic notions of isogeometric analysis, restricting ourselves to spaces constructed via locally quasi uniform $p$-open knot vectors as required by the theory presented in \cite{Beirao-da-Veiga_2014aa,Buffa_2018aa}.

\begin{definition}[B-Splines, \cite{Beirao-da-Veiga_2014aa}]\label{def::splines}
Fix $p$ and $k$ such that $0\leq p< k$. 
A \emph{locally quasi uniform $p$-open knot vector} is given by a set
\begin{align*}
\Xi = \big[{\xi_0 = \cdots =\xi_{p}}\leq \cdots \leq{\xi_{k}=\cdots =\xi_{k+p}}\big]\in[0,1]^{k+p+1}
\end{align*}
with $\xi_0 = 0$ and $\xi_{k+p}=1$
such that there exists a constant $\theta\geq 1$ such that for all $p\leq j < k$  one finds $\theta^{-1}\leq h_j\cdot h_{j+1}^{-1} \leq \theta,$ where  $h_j\coloneqq  \xi_{j+1}-\xi_{j}$ for all $\xi_j,\xi_{j+1}\in \Xi.$ 
The B-spline basis $ \lbrace b_j^p \rbrace_{0\leq j< k}$ is now defined by recursion as
\begin{align*}
b_j^p(x) & =\begin{cases}
\chi_{[\xi_j,\xi_{j+1})}&\text{ if }p=0,\\[8pt]
\frac{x-\xi_j}{\xi_{j+p}-\xi_j}b_j^{p-1}(x) +\frac{\xi_{j+p+1}-x}{\xi_{j+p+1}-\xi_{j+1}}b_{j+1}^{p-1}(x) & \text{ else,}
\end{cases}
\end{align*}
where $\chi_M$ denotes the indicator function for a set $M$. Moreover, we define the spline space  $S^p(\Xi)\coloneqq\operatorname{span}(\lbrace b_j^p\rbrace_{j <k}).$
\end{definition}

To obtain spline spaces in two spacial dimensions, define, for a tuple $\bb\Xi =(\Xi_1,$ $\Xi_2)$ and polynomial degrees $\bb p=(p_1,p_2)$ the spaces $S^{\bb p}(\bb \Xi)\coloneqq S^{p_1}(\Xi_1)\otimes S^{p_2}(\Xi_2)$. {For simplicity, we will assume {all interior knots} to have the same multiplicity.}
Given knot vectors $\Xi_1,$ $\Xi_2$ with knots {$\xi_{i}^k < \xi_{i+1}^k$ and $\xi_{i}^k,\xi^k_{i+1}\in \Xi_k$ for both, $k=1,2$, sets of the form $[\xi_{j}^1,\xi^1_{j+1}]\times[\xi^2_{j},\xi^2_{j+1}]$} will be called \emph{elements}. We reserve the letter $h$ for the maximal diameter of all elements.
{
We remark that this tensor product construction does not allow for local refinement. Approaches to omit this problem have been suggested, see e.g.~\cite{Evans_2014aa} and the sources cited therein, but are beyond the scope of this article.
}

Let $\square\coloneqq [0,1]^2$ denote the unit square. As usual in the framework of isogeometric analysis, the geometry $\Gamma=\bigcup_{j\leq N}\Gamma_j$ will be given as a family of mappings 
\begin{align}
\bb F_j\colon  \square \to \Gamma_j \subset \R^3,\label{def::geom}
\end{align}
which we will refer to as \emph{parametrization}.
These mappings will be given by NURBS mappings, 
i.e.,~by mappings with a representation
\begin{align*}
\bb F_j(x,y)\coloneqq \sum_{0\leq j_1<k_1}\sum_{0\leq j_2<k_2}\frac{\bb c_{j_1,j_2} b_{j_1}^{p_1}(x) b_{j_2}^{p_2}(y) w_{j_1,j_2}}{ \sum_{i_1=0}^{k_1-1}\sum_{i_2=0}^{k_2-1} b_{i_1}^{p_1}(x) b_{i_2}^{p_2}(y) w_{i_1,i_2}},
\end{align*}
for control points $\bb c_{j_1,j_2}\in \R^3$ and weights $w_{i_1,i_2}>0.$ For further concepts and algorithmic realization of the NURBS we refer to \cite{Piegl_1997aa}.

{We assume $\Gamma = \bigcup_{j\leq N} \Gamma_j$ to be {the} piecewise smooth boundary of a simply connected Lipschitz domain.} Moreover, we assume any mapping of the parametrization to be non-singular and invertible.
On any interface $\Gamma_j\cap \Gamma_i \neq \emptyset$ we require the involved mappings to coincide, i.e.,~$\bb F_j(\cdot,1) \equiv \bb F_i(\cdot,0)$ must be satisfied up to orientation of the reference domain. 

We remark that, as long as the assumptions stated above are fulfilled, the description of the geometry is independent of the analysis that will follow, i.e., one could choose sufficiently regular mappings that are not representable via NURBS, for example, mappings containing trigonometric functions.

\subsection{A Conforming Discretization}\label{sec::subsec::conforming}

Let $\bb F_j\colon \square\to \Gamma_j$ be a mapping of the parametrization of $\Gamma$.
Defining the \emph{surface measure} $\tau$ via
\begin{align}
\tau(\bb x)\coloneqq \norm{\partial_{x}\bb F_j( {\bb x})\times \partial_{y}\bb F_j( {\bb x})}_{\mathbb{R}^3},\quad \bb x\in \square,
\end{align}
the geometry transformations required for an analysis of isogeometric boundary element methods are of the form
\begin{align*}
\iota_0(\bb F_j)(f_0)(\bb x) \coloneqq {}&{}(f_0\circ\bb F_j)(\bb x),&&\bb x\in\square,\\
\iota_1(\bb F_j)(\bb f_1)(\bb x) \coloneqq {}&{}\big(\tau \cdot (d\bb F_j^\intercal)^{-1} (\bb f_1\circ \bb F_j)\big)(\bb x),&&\bb x\in\square,\\
\iota_2(\bb F_j)(f_2)(\bb x) \coloneqq {}&{}(\tau \cdot (f_2\circ \bb F_j))(\bb x),&&\bb x\in\square.
\intertext{
{We note that, since $d\bb F_j^\intercal$ is a rectangular matrix, $(d\bb F_j^\intercal)^{-1}$ is a common abuse of notation, whose meaning is discussed for example in \cite[Chapter~5.4]{Peterson_2006aa}. In short, under mild assumptions on the geometry mapping, $\iota_1$ has to be understood in the sense of mapping a tangential vector field on a two-dimensional manifold embedded into $\mathbb R^3$ to the tangential field of the two-dimensional reference domain.}
For implementations, this technicality can usually be omitted, since the operations on the reference domain merely require the push-forwards}
    (\iota_0(\bb F_j))^{-1}(f_0)(\bb x)={}&{} (f_0\circ \bb F_j^{-1})(\bb x),&&\bb x\in\Gamma_j,\\
    (\iota_1(\bb F_j))^{-1}(\bb f_1)(\bb x)={}&{} \left(\tau^{-1} \cdot (d\bb F_j)^\intercal( \bb f_1\circ \bb F_j^{-1})\right)(\bb x),&&\bb x\in\Gamma_j,\\
    (\iota_2(\bb F_j))^{-1}(f_2)(\bb x)={}&{} \left(\tau^{-1} \cdot (f_2\circ \bb F_j^{-1})\right)(\bb x),&&\bb x\in\Gamma_j,
\end{align*}
where the computation of the inverse $\bb F_j^{-1}$ is not required, since all discrete entities are known and constructed w.r.t.~the reference coordinates.

An important property of these geometry transformations is that the following diagram
$$ \begin{tikzcd}[row sep = 3em,column sep = 1.3cm]
H^1(\Gamma_j)\ar{d}[description]{\iota_0}\ar{r}[description]{\bcurl}& \bb H(\div_\Gamma,\Gamma_j)\ar{r}[description]{\div}\ar{d}[description]{\iota_1}&  L^2(\Gamma_j)\ar{d}[description]{\iota_2}\\
H^{1}(\square)\ar{r}[description]{{\bcurl_\Gamma}} & \bb H(\div,\square)\ar{r}[description]{\div_\Gamma} &  L^{{2}}(\square)
\end{tikzcd}$$
commutes \cite{Buffa_2018aa,Peterson_1995aa}. Thus, a conforming spline basis of $H^1(\square)$ yields automatically conforming finite dimensional discretization of the entire diagram. More precisely, given polynomial degrees $p_1,p_2>0$, the mapping properties of the differential operators yield the conforming spline spaces
\begin{align*}
\S^0_{\bb p,\bb\Xi}(\square)\coloneqq {}&{} S^{p_1,p_2}(\Xi_1,\Xi_2),&&\subset H^1(\square),\\
\bb \S^1_{\bb p,\bb\Xi}(\square)\coloneqq {}&{}  S^{p_1,p_2-1}(\Xi_1,\Xi_2') \times  S^{p_1-1,p_2}(\Xi_1',\Xi_2),&&\subset \bb H^0(\div,\square),\\
\S^2_{\bb p,\bb\Xi}(\square)\coloneqq {}&{} S^{p_1-1,p_2-1}(\Xi_1',\Xi_2')&&\subset L^2(\square),
\end{align*}
together with their mapped counterparts on the surface.

For the multipatch boundary $\Gamma = \cup_{j\leq N} \Gamma_j$ let $\bb \Xi \coloneqq (\bb \Xi_j)_{j\leq N}$ be an $N$-tuple of knot vectors as in Definition \ref{def::splines}. Let $\bb p=(\bb p_j)_{j\leq N}$ an $N$-tuple of pairs of integers $\bb p_j=\big(p_1^{(j)},p_2^{(j)}\big)$, corresponding to polynomial degrees for each patch $\Gamma _j$. Then we define the \emph{spline complex} on the boundary $\Gamma$ via
\begin{align*}
\S^0_{\bb p,\bb\Xi}(\Gamma)\coloneqq {}&{} \left\lbrace f\in H^{1/2}(\Gamma)\colon  \iota_0(\bb F_j)(f|_{\Gamma_j}) \in \S^0_{\bb p_j,\bb\Xi_j}(\square)\text{ for all }j\leq N\right\rbrace,\\
 \bb \S_{\bb p,\bb\Xi}^1(\Gamma)\coloneqq {}&{} \left\lbrace \bb f\in \bb H_\times^{{-1/2}}(\div_\Gamma,\Gamma)\colon  \iota_1(\bb F_j)(\bb f|_{\Gamma_j}) \in \bb \S_{\bb p_j,\bb\Xi_j}^1(\square)\text{ for all }j\leq N\right\rbrace,\\
 \S^2_{\bb p,\bb\Xi}(\Gamma)\coloneqq {}&{} \left\lbrace f\in H^{-1/2}(\Gamma)\colon  \iota_2(\bb F_j)(f|_{\Gamma_j}) \in \S^2_{\bb p_j,\bb\Xi_j}(\square)\text{ for all }j\leq N\right\rbrace.
\end{align*}
Throughout this paper, we will denote by $p$ the minimal polynomial degree used for the construction of $\S^0_{\bb p,\bb\Xi}(\Gamma)$.
\begin{remark}
In the spirit of the isogeometric paradigm, degrees and knot vectors of the discrete B-spline spaces can be chosen to match the properties of the geometry discretization \cite{Hughes_2005aa}. 
Note, however, that there is no theoretical requirement for $\bb p$ and $\bb \Xi$ to match the discretization of the geometry if we assume sufficient regularity of the parametrization. This fact will be used later on to benchmark different orders of discretization on the same geometry. 
\end{remark}

By definition of the spline spaces above, the sequence 
\begin{equation}
\begin{tikzcd}
\S^0_{\bb p,\bb\Xi}(\Gamma)\ar{r}{\bcurl_\Gamma}& 
\bb \S_{\bb p,\bb\Xi}^1(\Gamma)\ar{r}{\div_\Gamma}&
\S^2_{\bb p,\bb\Xi}(\Gamma)
\end{tikzcd}\label{spline::sequence}
\end{equation}
is a conforming multipatch discretization of the two-dimensional sequence 
\begin{equation}
\begin{tikzcd}
H^{1/2}(\Gamma)\ar{r}{\bcurl_\Gamma}& 
\chispace \ar{r}{\div_\Gamma}&
H^{-1/2}(\Gamma).
\end{tikzcd}
\end{equation}
We refer to \cite{Buffa_2018aa} for an in-depth discussion on how these spline spaces on the boundary are connected with the B-spline discretization of the three-dimensional de Rham sequence.

Replacing $\chispace$ by $\bb\S^1_{\bb p, \bb \Xi}(\Gamma)\subset\chispace$ yields the discrete variational problem to \eqref{problem::variational::cont}, given as follows.
\begin{problem}[Discrete Propblem]
Find $\bb w_h\in\bb \S^1_{\bb p, \bb \Xi}(\Gamma)$ such that
\begin{align}
\langle\MSL_\kappa\bb w_h, \bb \mu_h\rangle_\tau=-\langle\bb \gamma_t^+\bb e_i,\bb\mu_h\rangle_\tau,\label{problem::variational::disc}
\end{align}
for all $\bb\mu_h\in \bb \S^1_{\bb p, \bb \Xi}(\Gamma)$.
\end{problem}
Given a basis $\{\bb\varphi_i\}_{i=1}^N$ of $\bb \S_{\bb p,\bb\Xi}^1(\Gamma)$,
this yields the linear system
\begin{align}\label{eq:linsys}
\MSL_{\kappa,h}\bb w_h=-\bb f_h,
\end{align}
where the right-hand side $\bb f_h$ is given by
$\big[\bb f_h\big]_j=\langle\bb\gamma_t^+\bb e_i,\bb\varphi_j\rangle_\times$
and the system matrix $\MSL_{\kappa,h}$ by
\begin{align}\label{eq:MSLtested}
\begin{aligned}
\big[\MSL_{\kappa,h}\big]_{i,j}
={}&
\langle\MSL_\kappa\bb\varphi_j, \bb\varphi_i\rangle_\times\\
={}&
\int_\Gamma\int_\Gamma G_{\kappa}(\bb x-\bb y)\bb\varphi_j(\bb x)\cdot\bb\varphi_i(\bb y)\dd\sigma_{\bb y}\dd\sigma_{\bb x}
\\
&\qquad-\frac{1}{\kappa^2}\int_\Gamma\int_\Gamma G_{\kappa}(\bb x-\bb y)\div_\Gamma\bb\varphi_j(\bb x)\div_\Gamma\bb\varphi_i(\bb y)\dd\sigma_{\bb y}\dd\sigma_{\bb x},
\end{aligned}
\end{align}
see also \cite{Buffa_2003ab}. We remark that the system matrix is symmetric, but not Hermitian.

\subsection{Approximation Properties and Discrete Inf-Sup Condition}\label{sec::subsec::approx}

The conforming spline spaces introduced in the previous section provide approximation results of optimal order in $\chispace$, w.r.t.~patchwise regularity. Therefore, setting $s\geq 0$, we define the patchwise norms
\[
	\norm{\bb f}_{\tilde {\bb H}^s(\Gamma)} \coloneqq \sum_{j\leq N} \norm{\bb f}_{\bb H^s(\Gamma_j)},\qquad
	\norm{\bb g}_{\tilde {\bb H}^s(\div_\Gamma,\Gamma)} \coloneqq \sum_{j\leq N} \norm{\bb g}_{\bb H^s(\div_\Gamma,\Gamma_j)},
\]
for all functions $\bb f\in \bb L^2(\Gamma)$ and $\bb g \in \chispace$ for which these expressions are well defined. 
The corresponding spaces of higher patchwise regularity are defined canonically as subspaces of $\bb L^2(\Gamma)$
and $\bb H(\div_\Gamma,\Gamma)$ with finite norm, see \cite{Buffa_2018aa}.
\begin{theorem}[Approximation Properties of $\bb \S^1_{\bb p,\bb \Xi}(\Gamma)$, \cite{Buffa_2018aa}]\label{thm::hdiv}
    Let $\bb f\in \tilde{\bb H}^s(\div_\Gamma,\Gamma)$, $0\leq s\leq p$ and denote by $\bb f_h$ the $\bb H_\times^{-1/2}(\div_\Gamma,\Gamma)$-orthogonal projection of $\bb f$ onto $\bb\S^1_{\bb p,\bb \Xi}(\Gamma)$.
    Then one finds 
    \begin{align*}
        \norm{\bb f-\bb f_h}_{\bb H_\times^{-1/2}(\div_\Gamma,\Gamma)} \lesssim h^{1/2+s}  \norm{\bb f}_{\tilde {\bb H}^s(\div_\Gamma,\Gamma)}.
    \end{align*}
\end{theorem}

According to the classical theory of the electric field integral equation the following holds.

\begin{lemma}[Criteria for a Stable Discretization, {{\cite[Sec.~3]{Bespalov_2010aa}}, \cite[Prop.~4.1]{Buffa_2003aa}}]\label{theLemmaThatGivesInfSup}
	Under the assumptions that
	\begin{enumerate}
		\item there exists a continuous splitting $\chispace = \bb W\oplus \bb V$ such that the bilinear form induced by the variational formulation \eqref{problem::variational::cont} is stable and coercive on $\bb V\times \bb V$ and $\bb W\times \bb W$, and compact on $\bb V\times \bb W$ and $\bb W\times \bb V$,
		\item $\bb \S^1_{\bb p,\bb \Xi}(\Gamma)$ can be decomposed into a sum $\bb \S^1_{\bb p,\bb \Xi}(\Gamma)\coloneqq \bb W_h \oplus \bb V_h$ of closed subspaces of $\chispace$,
		\item $\bb W_h$ and $\bb V_h$ are stable under complex conjugation, and
		\item it holds that $\bb W_h\subseteq \bb W$, as well as the so-called \emph{gap-property}
			\begin{align}\label{eq::gap-property}
		 	\sup_{\bb v_h\in \bb V_{h}}\inf_{\bb v\in \bb V}\frac{\norm{\bb v-\bb v_h}_{\chispace}}{\norm{\bb v_h}_{\chispace}}\stackrel{h\rightarrow 0}{\longrightarrow} 0,
			\end{align}
	\end{enumerate}
	the discrete problem \eqref{problem::variational::disc} enjoys $\inf$-$\sup$-stability.
\end{lemma}

The continuous splitting of $\chispace$ has been discussed in the literature, see, e.g., \cite{Buffa_2003ab}, and is required for proving the $\inf$-$\sup$-stability of \eqref{problem::variational::cont}. 
One of the most concise (although not self-contained) constructions of said splitting and the discrete inf-sup condition according to this scheme is due to \cite{Bespalov_2010aa}, whose lines we will follow closely, starting with the introduction of some necessary operators. The theory behind them goes back to \cite{Hiptmair_2002aa}.

\begin{lemma}[Regularising Projection]\label{lem::regularizingprojection}
	For compact domains $\Omega$ with Lipschitz boundary there exists a continuous projection $\mathsf R\colon \chispace\to{\bb H_\times^{1/2}(\Gamma)}$ such that
	\begin{align}\label{eq:diffRcomm}
	(\div_\Gamma \circ \mathsf R)( \bb u) &=\div_\Gamma \bb u,
	\end{align}
	{for all $\bb u\in\chispace$.}
\end{lemma}
\begin{proof}
	The definition of a suitable operator is done in \cite[Lem.~3.1]{Bespalov_2010aa}, which we shortly recap for later reference.
	
	{For any $\bb v\in\chispace$, the solution of the Neumann problem
	\begin{align}\label{eq:neumannprob}
	\begin{aligned}
	\Delta w ={}&0&&\text{in}~\Omega,\\
	\bb\grad\, w\cdot\bb n ={}&\div_{\Gamma}\bb v&&\text{on}~\Gamma,
	\end{aligned}
	\end{align}
	defines a field $\bb\grad\, w\in\bb H^0(\div 0,\Omega)$ with $\langle\bb\grad\, w\cdot\bb n,1\rangle_{L^2(\Gamma)}=0$. Using a continuous lifting operator $\mathsf{L}\colon\bb H^0(\div 0,\Omega)\to\bb H^1(\Omega)$ with $\bb\curl\mathsf{L}\bb u=\bb u$ for all $\bb u\in\bb H^0(\div 0,\Omega)$ satisfying $\langle\bb u\cdot\bb n,1\rangle_{L^2(\Gamma)}=0$, see, e.g., \cite[Theorem 3.4]{GR86}, we finally arrive at $\mathsf{R}\bb v\coloneqq\bb\gamma_t^-\mathsf{L}\bb\grad\, w\in\bb H^{1/2}_\times(\Gamma)$ and \eqref{eq:diffRcomm} follows. Since $\bb\grad\, w$ depends continuously on $\div_\Gamma\bb v$ and $\bb\gamma_t^-\colon\bb H^1(\Omega)\to\bb H^{1/2}_\times(\Gamma)$ is also continuous, the continuity of $\mathsf{R}$ follows.
	}
\end{proof}

Indeed, one can show the continuous $\inf$-$\sup$-condition via the splitting $\bb V\coloneqq\mathsf {R}\big( \chispace\big)$ and $\bb W\coloneqq(\operatorname{Id}-\mathsf R)\big(\chispace\big)$.

The construction of a corresponding discrete splitting relies on the multipatch interpolation operators introduced in \cite{Buffa_2018aa}, given by 
\begin{align*}
\tilde\Pi^0_\Gamma&\colon &\hspace{-2.8cm}H^{1/2}(\Gamma)\supseteq \mathcal{D}\big(\tilde\Pi^0_\Gamma\big)\to{}&{}\S_{\bb p,\bb\Xi}^0(\Gamma),\\
\bb{\tilde\Pi}^1_\Gamma&\colon&\hspace{-2.8cm}\chispace\supseteq\mathcal{D}\big(\bb{\tilde\Pi}^1_\Gamma\big) \to{}&{}\bb\S_{\bb p,\bb\Xi}^1(\Gamma),\\
\tilde\Pi^2_\Gamma&\colon &\hspace{-2.8cm}H^{-1/2}(\Gamma)\supseteq \mathcal{D}\big(\tilde\Pi^2_\Gamma\big)\to{}&{}\S_{\bb p,\bb\Xi}^2(\Gamma),
\end{align*}
with domains $\mathcal{D}(\,\cdot\,)$. Note that these projections commute with the surface differential operators $\bb \curl_\Gamma$ and $\div_\Gamma,$ i.e.,~one finds
\begin{align}
	\begin{aligned}
	(\bb \curl_\Gamma \circ \tilde\Pi^0_\Gamma )(f) &= (\bb {\tilde\Pi}^1_\Gamma\circ\bb\curl_\Gamma)(f),\\
	(\div_\Gamma \circ \,\bb {\tilde\Pi}^1_\Gamma )(\bb f) &= (\tilde\Pi^2_\Gamma\circ\div_\Gamma) (\bb f).
	\end{aligned}
	\label{eq::commutingpis}
\end{align}
Among other estimates about these interpolation operators, \cite{Buffa_2018aa} provides the following.
\begin{lemma}\label{lemma::interpolationlemma}
	Let $\bb f\in \bb{\tilde H}^s(\Gamma)$ for $1\leq s\leq p$. Then it holds that 
	\begin{align*}
		\norm{\bb f-\bb{\tilde\Pi}^1_\Gamma \bb f}_{\bb L^2(\Gamma)}&\lesssim h^s\norm{\bb f}_{\bb {\tilde H}^s(\Gamma)}.
	\end{align*}
\end{lemma}

{
Other than in \cite{Bespalov_2010aa}, we cannot use the operator $\mathsf{R}$ to introduce a discrete splitting, since the image of $\mathsf{R}$ is not patch-wise in $\bb H^1$, which would be required to be contained in $\mathcal{D}\big(\bb{\tilde\Pi}^1_\Gamma\big)$. Instead, we have to introduce another regularising projection.
\begin{lemma}[Regularising Projection for higher Regularity]\label{lem::regularizingprojectionspline}
	For compact domains $\Omega$ with Lipschitz boundary there exists a continuous projection $\mathsf R_0\colon\bb H^0(\div_{\Gamma},\Gamma)\to\tilde{\bb H}^1(\Gamma)$ such that
	$(\div_\Gamma \circ \mathsf R_0)( \bb u) =\div_\Gamma \bb u,$
	for all {$\bb u\in\bb H^0(\div_{\Gamma},\Gamma)$}.
\end{lemma}
\begin{proof}
The proof is in analogy to Lemma~\ref{lem::regularizingprojection}. First, we remark that $\div_{\Gamma}\bb u\in L^2(\Gamma)$. Thus, \eqref{eq:neumannprob} yields a field $\bb\grad\, w\in \bb H^{1/2}(\div 0,\Omega)$ with $\langle\bb\grad\, w\cdot\bb n,1\rangle_{L^2(\Gamma)}=0$. \cite[Remark 3.12]{GR86} shows that there is a continuous lifting operator $\mathsf{L}_{1/2}\colon\bb H^{1/2}(\div 0,\Omega)\to\bb H^{3/2}(\Omega)$ with $\bb\curl\mathsf{L}\bb u=\bb u$ for all $\bb u\in\bb H^{1/2}(\div 0,\Omega)$ satisfying $\langle\bb u\cdot\bb n,1\rangle_{L^2(\Gamma)}=0$. This yields the assertion by patchwise application of the rotated tangential trace.
We remark that the continuity of $\mathsf{L}_{1/2}$ follows by noting that the construction of the extensions in \cite[Theorem 3.4]{GR86} and \cite[Corollary 3.3]{GR86} depend continuously on the input data. The interpolation argument of \cite[Remark 3.12]{GR86} then yields the continuity assertion, since the image of the procedure in \cite[Theorem 3.4]{GR86} and \cite[Corollary 3.3]{GR86} coincides for equal input data in terms of their respective equivalence classes.
\end{proof}
}

{In analogy to the continuous setting, by} Lemma \ref{lem::regularizingprojectionspline} and the construction and properties of the quasi-interpolation operators constructed in \cite{Buffa_2018aa}, the definition of the discrete splitting via
\begin{align*}
\bb V_h\coloneqq (\bb{\tilde\Pi}_\Gamma^1\mathsf \circ\, {\mathsf R_0})\big(\bb \S^1_{\bb p,\bb \Xi}(\Gamma)\big),\qquad \bb W_h\coloneqq (\Id - \bb{\tilde\Pi}_\Gamma^1\circ{\mathsf R_0})\big(\bb \S^1_{\bb p,\bb \Xi}(\Gamma)\big),
\end{align*}
{is well defined and} would be a suitable candidate to fulfill the assumptions of Lemma~\ref{theLemmaThatGivesInfSup}.
{
\begin{remark}\label{rem::whinw}
	The construction of both $\mathsf R$ and $\mathsf R_0$ make it clear that the kernel of the respective operator consists exactly of the divergence free functions. Thus, it follows that $\bb W_h\subseteq \bb W$ holds, compare \cite[Eq.~3.5]{Bespalov_2010aa}.	
\end{remark}
}

We are now ready to provide a statement about the $\inf$-$\sup$-stability of the discretized EFIE.

\begin{theorem}\label{lem::disc_inf-sup}
	{The discrete problem \eqref{problem::variational::disc} enjoys $\inf$-$\sup$-stability.}
\end{theorem}
\begin{proof}
	{First, we consider the case of $\bb\S^1_{\bb p,\bb \Xi}(\Gamma)\subset \bb H^{1/2}_\times(\Gamma),$ i.e., when $\bb\S^1_{\bb p,\bb \Xi}(\Gamma)$ is patchwise continuous.
		Due to Remark \ref{rem::whinw}, it remains to check the gap property \eqref{eq::gap-property} for $\bb V$ and $\bb V_h$.}
	We write
	\begin{align*}
	\sup_{\bb v_h\in \bb V_{h}}\inf_{\bb v\in \bb V}\frac{\norm{\bb v-\bb v_h}_{\chispace}}{\norm{\bb v_h}_{\chispace}} 
	\lesssim{}&{}\sup_{\bb v_h\in \bb V_{h}}\frac{\norm{{\mathsf R_0} \bb v_h-\bb v_h}_{\chispace}}{\norm{\bb v_h}_{\chispace}}\\
	={}&{} \sup_{\bb v_h\in \bb V_{h}}\frac{\norm{({\mathsf R_0} -\bb{\tilde\Pi}^1_\Gamma \circ {\mathsf R_0})(\bb  v_h)}_{\chispace}}{\norm{\bb v_h}_\chispace}.
	\end{align*}
	Note that the last equality holds because one can show that $\tilde{\bb\Pi}^1_\Gamma \circ {\mathsf R_0}$ is a projection, {as done in \cite[Sec.~6]{Bespalov_2010aa} for $\mathsf R$.} Thus, it holds that
	\[
	(\tilde{\bb\Pi}^1_\Gamma \circ {\mathsf R_0}) (\bb V_h )= \bb V_h \coloneqq (\tilde{\bb\Pi}^1_\Gamma \circ {\mathsf R_0})(\bb \S^1_{\bb p,\bb \Xi}(\Gamma)).
	\]
	{Since {the canonical embedding $\bb H^0(\div_{\Gamma},\Gamma)\hookrightarrow \chispace$ is continuous}, 
	we arrive at  
	$$	\norm{(\mathsf R_0 -\bb{\tilde\Pi}^1_\Gamma \circ \mathsf R_0)(\bb  v_h)}_{\chispace}\lesssim \norm{(\mathsf R_0 -\bb{\tilde\Pi}^1_\Gamma \circ \mathsf R_0)(\bb  v_h)}_{H^0(\div_\Gamma,\Gamma)}.$$
	By the divergence preserving property of $\mathsf R_0$ and the fact that the interpolation operators are projections which commute w.r.t.~the surface differential operator, we can apply 
	Lemma~\ref{lemma::interpolationlemma} arriving at
	\[
	\norm{(\mathsf R_0 -\bb{\tilde\Pi}^1_\Gamma \circ \mathsf R_0)(\bb  v_h)}_{\bb H^0(\div_\Gamma,\Gamma)} = \norm{(\mathsf R_0 -\bb{\tilde\Pi}^1_\Gamma \circ \mathsf R_0)(\bb  v_h)}_{\bb L^2(\Gamma)} \lesssim h\norm{\mathsf R_0\bb  v_h}_{\bb {\tilde H}^1(\Gamma)}.
	\]
	Note that the right hand side is well defined due to $\mathsf R_0 (\bb V_h) \subset \bb {\tilde H}^1(\Gamma).$
	Combining the above with the continuity of the $\mathsf R_0$ operator and inverse estimates, cf.~Lemma~\ref{lem:inverseestimate}, yields
	\[
	\norm{(\mathsf R_0 -\bb{\tilde\Pi}^1_\Gamma \circ \mathsf R_0)(\bb  v_h)}_{\chispace} \lesssim h^{1/2}\|\bb v_h\|_{\chispace},
	\]
	and thus the assertion for the case $\bb\S^1_{\bb p,\bb \Xi}(\Gamma)\subset \bb H^{1/2}_\times(\Gamma)$. The case $\bb\S^1_{\bb p,\bb \Xi}(\Gamma)\not\subset \bb H^{1/2}_\times(\Gamma)$, which is realized only for maximal knot repetition within knot vectors, reduces to the classical theory of higher order Raviart Thomas elements on quadrilaterals, cf.~\cite{Buffa_2003aa,Zaglmayr_2006aa}.
	}
\end{proof}

Following classical Babu\v{s}ka-Brezzi theory \cite{Babuska_1969aa,Xu_2002aa}, we can finally combine Theorems \ref{thm::hdiv} and \ref{lem::disc_inf-sup} and arrive at the main result of this section.

\begin{theorem}[Discretization Error]\label{thm::quasioptimality}
	The solution to \eqref{problem::variational::disc} exists and is unique.
	
	Moreover, assuming $\bb w\in\tilde{\bb H}^{s}(\div_\Gamma,\Gamma)$ for some $0<s\leq p$, for the solutions $\bb w\in \bb H^{-1/2}_\times(\div_\Gamma,\Gamma)$ and $\bb w_h\in \bb \S^1_{\bb p,\bb \Xi}(\Gamma)$ of \eqref{problem::variational::cont} and \eqref{problem::variational::disc} we find that
	\begin{align*}
	\norm{\bb w-\bb w_h}_{\bb H^{-1/2}_\times(\div_\Gamma,\Gamma)}\lesssim h^{s+1/2}\norm{\bb w}_{\tilde{\bb H}^{s}(\div_\Gamma,\Gamma)}.
	\end{align*} 
\end{theorem}

As a corollary, we can predict the expected convergence rates of the
scattered electric field. Similar to scalar-valued problems,
the convergence rate of the field doubles.
\begin{corollary}\label{cor:potentialconv}
	Let $\bb x\in\Omega^c$ fixed. Let $\bb w$ be the solution to \eqref{problem::variational::cont} and $\bb w_h$ the solution
	to the numerical problem \eqref{problem::variational::disc}.
	Then, for $\bb e_s=\tilde{\MSL}_\kappa\bb w$ and $\bb e_{s,h}=\tilde{\MSL}_\kappa\bb w_h,$ it holds
	\[
	\|\bb e_s(\bb x)-\bb e_{s,h}(\bb x)\|_{\mathbb{C}^3}
	\leq C(\bb x)h^{2p+1}
	\|\bb w\|_{\tilde{\bb H}^{p}(\div_\Gamma,\Gamma)},
	\]
	if $\bb w$ and the solution of a suitable adjoint problem are sufficiently smooth.
\end{corollary}
\begin{proof}
	One readily verifies that $\big(\tilde{\MSL}_{\kappa}\cdot\big)(\bb x)$ is a linear
	and continuous functional on $\chispace$ for given $\bb x$. Let then $\bb\varphi^{(\bb
	x)}$ be the solution of the adjoint problem of finding $\bb \varphi^{(\bb x)}\in \chispace$ such that
		\begin{align}\label{eq:adjointprob}
			\big\langle \MSL_\kappa\bb\xi, \bb\varphi^{(\bb x)}\big\rangle_\times = \big(\tilde{\MSL}_{\kappa}\bb \xi\big)(\bb x)
		\end{align}
		holds for all $\bb\xi \in \chispace$.
	Let $\bb\varphi^{(\bb x)}_h$ denote its discrete analogon.
	The assertion now follows by applying a standard argument, see also
	\cite[Theorem 4.2.14]{SS11}, to each component of the scattered field
	to obtain
	\[
	\|\bb e_s(\bb x)-\bb e_{s,h}(\bb x)\|_{\mathbb{C}^3}
	\lesssim
	\norm{\bb w-\bb w_h}_{\bb H^{-1/2}_\times(\div_\Gamma,\Gamma)}
	\norm{\bb \varphi^{(\bb x)}-\bb \varphi^{(\bb x)}_h}_{\bb H^{-1/2}_\times(\div_\Gamma,\Gamma)}.
	\]
	The previous theorem yields the assertion with $C(\bb x)=C\|\bb\varphi^{(\bb x)}\|_{\tilde{\bb
	H}^{p}(\div_\Gamma,\Gamma)}$, if the solutions to \eqref{problem::variational::cont} and
	\eqref{eq:adjointprob} are smooth enough.
\end{proof}

\begin{remark}
	The proof applies to any linear and continuous output functional of $\bb w$. Thus,
	similar error estimates hold also for other quantities of interest, for
	example for path integrals of the electric field, i.e., voltages, {or radar cross sections}, cf.~\cite{Jackson_1998aa}.
\end{remark}

%% file: sec_disc.tex

This section is concerned with the {implementation of a specialized fast method for the
electric field integral equation} with conforming B-spline spaces
in the isogeometric framework \eqref{eq:linsys}. The assembly of the system
matrices will be discussed in the following subsection. Since, in
general, the matrices will be densely populated, Section \ref{sec:FMM}
discusses a fast compression method based on interpolation of
the kernel function as introduced in \cite{Dolz_2016aa,Dolz_2018aa}.
A particular feature of this approach is that its simple implementation
integrates effortlessly into the $\mathcal{H}^2$-matrix framework
\cite{Borm_2010aa}, which is a more efficient variant of the
commonly used $\mathcal{H}$-matrix framework \cite{Hac15}. Section
\ref{sec:FMMerror} is then concerned with the error analysis of the
method.

Before we start, we briefly comment on the adaptive cross
approximation, which has been used in \cite{Simpson_2018aa}. Being
a simple and purely algebraic algorithm, it is very popular in the
engineering community \cite{Kurz_2002aa}. However, due to the overlapping supports
of higher-order B-spline discretizations, it will
inevitably lead to redundant kernel and geometry evaluations,
and will thus become increasingly inefficient for higher order.
In contrast, the scheme we are going to present in this section
avoids such redundant computations by construction.

\subsection{Assembly of the System Matrix}
We assume that the B-spline space
$\bb \S_{\bb p,\bb\Xi_m}^1(\Gamma)$ is, on each patch $\Gamma _j$, generated by the
tuple $\bb\Xi _{m,j}=(\Xi_m,\Xi_m)$, where $\Xi_m$ is an equidistant knot vector with $2^m$,
$m\geq 0$, elements. This corresponds to $m$ steps of uniform refinement in terms
of the reference domain and generates a nested sequence of meshes. Then, for
each level of refinement $m$, the mesh consists of $4^m$ elements per patch.
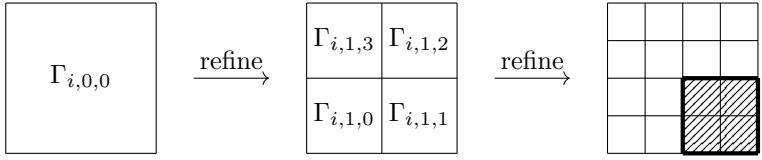
\begin{figure}
	\centering
		\begin{tikzpicture}
		\draw (0,0) -- (0,2);
		\draw (0,2) -- (2,2);
		\draw (2,2) -- (2,0);
		\draw (2,0) -- (0,0);
		\node (C1) at (1,1) {$\Gamma_{i,0,0}$};
		
		\node (A) at (2.5,1) {};
		\node (B) at (3.5,1) {};
		\draw [->] (A.center) -- (B.center) node [midway, above] {refine};
		
		\draw (0+4,0) -- (0+4,2);
		\draw (0+4,2) -- (2+4,2);
		\draw (2+4,2) -- (2+4,0);
		\draw (2+4,0) -- (0+4,0);
		\draw (0+4,1) -- (2+4,1);
		\draw (0+4+1,0) -- (0+4+1,2);
		\node (C21) at (1+4-.5,1-.5) {$\Gamma_{i,1,0}$};
		\node (C22) at (1+4+.5,1-.5) {$\Gamma_{i,1,1}$};
		\node (C23) at (1+4+.5,1+.5) {$\Gamma_{i,1,2}$};
		\node (C24) at (1+4-.5,1+.5) {$\Gamma_{i,1,3}$};
		
		\node (C) at (6.5,1) {};
		\node (D) at (7.5,1) {};
		\draw [->] (C.center) -- (D.center) node [midway, above] {refine};
		
		\path [pattern = north east lines] (0+9,0) rectangle (0+9+1,1);
		
		\draw (0+8,0) -- (0+8,2);
		\draw (0+8,2) -- (2+8,2);
		\draw (2+8,2) -- (2+8,0);
		\draw (2+8,0) -- (0+8,0);
		\draw (0+8,1) -- (2+8,1);
		\draw (0+8+1,0) -- (0+8+1,2);
		\draw (0+8+.5,0) -- (0+8+.5,2);
		\draw (0+8+1.5,0) -- (0+8+1.5,2);
		\draw (0+8,1.5) -- (2+8,1.5);
		\draw (0+8,.5) -- (2+8,.5);

		\draw[ultra thick] (0+9,0) -- (0+9,1);
		\draw[ultra thick] (0+9+0,0) -- (0+9+1,0);
		\draw[ultra thick] (0+9+1,1) -- (0+9,1);
		\draw[ultra thick] (0+9+1,1) -- (0+9+1,0);
		
		\end{tikzpicture}
		\caption{Refinement of the patch induced by the $i$-th mapping. Bold region corresponds to cluster $\Gamma_{\bb \lambda}$ with $\bb \lambda = (i,1,1)$.}

	\label{fig::refinement}
\end{figure}
The key point of this refinement strategy is that it induces a quadtree structure on the geometry, cf.~Figure \ref{fig::refinement}, which we will use for our compression scheme. Each element $\Gamma_{i,j,k}$ within the nested sequence of meshes will be refered to by a tuple $(i,j,k)\eqqcolon \bb\lambda$, where $i$ denotes the corresponding parametric mapping, $j$ showcases the level of refinement of the element and $k$ denotes the index of the element in hierarchically order. For notational purposes, we will define $\abs{\bb \lambda}\coloneqq j$ and also introduce the
diffeomorphisms $\bb F_{\bb\lambda}\colon\square\to\Gamma_{\bb\lambda}$ which can easily
be defined by combining $\bb F_i$ with a suitable affine transformation.
For the efficient compression, each instance of $\Gamma_{\bb\lambda}$ is also considered as a \emph{cluster}, in the sense that $\Gamma_{\bb\lambda}$ will be considered as the set of tree leaves appended to the subtree with root $\Gamma_{\bb\lambda}.$ Na\"ively said, $\Gamma_{\bb \lambda}$ can be visualised as ``a square region on the geometry''.
The hierarchically ordered collection of all $\Gamma_{\bb\lambda}$ will be called \emph{cluster tree} and denoted by $\mathcal T$.

For each pair of clusters in $\mathcal T$, the fundamental solution $G_\kappa$ from
\eqref{eq:HelmholtzGreen} can be localized to a \emph{localized kernel function}
\begin{align}
G_{\kappa,\bb\lambda,\bb \lambda'} \colon \square\times \square & \to\mathbb C,
\qquad G_{\kappa,\bb\lambda,\bb \lambda'}(\bb s,\bb t)=G_\kappa\big(\bb F_{\bb\lambda}(\bb s)-\bb F_{\bb\lambda'}(\bb t)\big)\label{eq:lockernel}
\end{align}
which reparametrizes the fundamental solution to $\square\times\square$.
This reduces the dimension (in terms of input variables) of the fundamental solution artificially.

On each element $\Gamma_{\bb\lambda}$, ansatz functions $\varphi_{\bb\lambda}$ can
be defined by lifting suitable shape functions $\hat\varphi$ on $\square$ to the surface by {the suitable (localized) pullback, thus defining $\varphi _{\bb\lambda}$}.
To define suitable shape functions of polynomial degree $p$ on $\square$, we introduce the knot vector $\Xi_m^*$, which is generated from $\Xi_m$ by increasing the multiplicity of each knot to $p+1$. We then define the spaces $\S_{p,m}^*(\square)$, to be the discontinuous spaces generated by $\bb p=(p,p)$ and $\bb\Xi_m^*=(\Xi_m^*,\Xi_m^*)$. Then, for the particular case $m=0$, $\S_{p,0}^*(\square)$ contains all tensorised polynomials of degree $p$ on $\square$. Later, we will also require $\S_{p,m}^*(\square)$, $m>0$, which is generated by tensorised polynomials of degree $p$ on every element on the unit square.
The span of all ansatz functions $\varphi_{\bb\lambda}$ with $|\bb\lambda|=m$ then
yields a global discrete discontinuous function space $\S^*_{p,m}(\Gamma)$ of dimension $k\coloneqq 2^{2m}N(p+1)^2$.

Since B-splines are piecewise polynomials, it clearly holds that
\[
\bb \S_{\bb p,\bb\Xi_m}^1(\Gamma)\subseteq\bb\S^*_{p,m}(\Gamma):=\S^*_{p,m}(\Gamma)\times\S^*_{p,m}(\Gamma),
\]
with $\bb p=(p,p)$ and $\bb\Xi_m=(\Xi_m,\Xi_m)$. We can
therefore represent each basis function of $\bb \S_{\bb p,\bb\Xi_m}^1(\Gamma)$
by a linear combination of basis functions of $\bb\S^*_{p,m}(\Gamma)$. This
yields a transformation matrix $\bb T$, which transforms the coefficient vector of a
function in $\bb \S_{\bb p,\bb\Xi_m}^1(\Gamma)$ to the coefficient vector of the
corresponding function in $\bb\S^*_{p,m}(\Gamma)$. Then, instead of assembling
the system of linear equations \eqref{eq:linsys} with respect to
$\bb\S_{\bb p,\bb\Xi_m}^1(\Gamma)$, one may assemble it with respect to
$\bb\S^*_{p,m}(\Gamma)$ to obtain a system matrix $\bb V_{\kappa,h}^*$
and a vector $\bb f_h^*$. A linear system of equations equivalent to \eqref{eq:linsys}
is then given by
\begin{align}
\bb T^\intercal \bb V_{\kappa,h}^* \bb T \bb w = -\bb T^\intercal \bb f_h^*.\label{eq::superspace}
\end{align}
Since the dimension of $\bb\S^*_{p,m}(\Gamma)$ is larger than the
dimension of $\bb\S_{\bb p,\bb\Xi_m}^1(\Gamma)$, the matrix $\bb V_{\kappa,h}^*$ is
larger than the matrix $\bb V_{\kappa,h}$. However, it has been shown in \cite{WegglerMatrix} for the case of classical higher order Raviart-Thomas elements
that the superspace approach can achiever better compression rates and, thus, better
computation times. In this particular case, the non-zero elements in $\bb T$ were
either $1$ or $-1$. In \cite{Dolz_2016aa, Dolz_2018aa}, the superspace approach has been
applied to represent higher order B-spline spaces for Laplace and Helmholtz problems,
where the elements of $\bb T$ were the coefficients of a suitable basis transformation.
Thus, the superspace approach in \eqref{eq::superspace} can be implemented as a mixture
of the two: Whereas, on each patch $\Gamma _j$, one can use the approach of
\cite{Dolz_2018aa} to find a suitable transformation matrix between
$\S_{p,m}^*(\Gamma_j)\times\S_{p,m}^*(\Gamma_j)$ and
$\S _{\bb p, \bb\Xi _{m}}^1(\Gamma)|_{\Gamma _j}$, one can use the approach of
\cite{WegglerMatrix} to enforce continuity across patch boundaries. The transformation
matrix $\bb T$ can then be seen as the product of two sparse matrices.

\begin{remark}
From an implementation point of view, the transformation matrix between
$\S_{p,m}^*(\Gamma_j)\times\S_{p,m}^*(\Gamma_j)$ and
$\S _{\bb p, \bb\Xi _{m}}^1(\Gamma)|_{\Gamma _j}$
can easily be constructed in a black-box fashion by exploiting the tensor product structure of
the two spaces and spline-interpolation in one dimension.
\end{remark}

The highly local support of the ansatz functions in $\bb\S^*_{p,m}(\Gamma)$ {has
several advantages. First, the numerical integration for the evaluation of the
matrix entries can be done with standard quadrature methods for higher order boundary
element methods, see \cite{SS97} or \cite{Harbrecht_2001aa}. Second,
it} will allow us to employ a version of the fast multipole method for the
matrix compression which perfectly fits the framework of isogeometric analysis. 
Of course, one may also use any other compression method to approximate $\bb V_{\kappa,h}^*$,
but we will see that our version of the fast multipole method in combination with
the structure of the isogemoetric mappings directly fits into the efficient
$\mathcal{H}^2$-matrix framework. Other compression methods tailored to
isogeometric mappings, but in the lowest-order context and in the less efficient
$\mathcal{H}$-matrix framework, have been compared in \cite{Harbrecht_2013ab}.

Before we introduce the compression scheme, we first
have to pull the matrix represention \eqref{eq:MSLtested}
back to the reference domain. According to \cite{Peterson_2006aa}, for two
basis functions $\bb\varphi_i$ and $\bb\varphi_j$
of $\bb\S^*_{p,m}(\Gamma)$ supported on $\Gamma_{\bb\lambda(i)}$ and $\Gamma_{\bb\lambda(j)}$, the first integral is given by
\begin{align}\label{eq:MSLref1}
\begin{aligned}
&\int_\Gamma\int_\Gamma G_k(\bb x-\bb y)\bb\varphi_i(\bb x)\cdot\bb\varphi_j(\bb y)\dd\sigma_{\bb y}\dd\sigma_{\bb x}\\
&{}\qquad\qquad=\int_\square\int_\square G_{\kappa,\bb\lambda(i),\bb\lambda(j)}(\bb s,\bb t)\hat{\bb\varphi}_j(\bb s)^{\intercal}d\bb F_{\bb\lambda(j)}(\bb s)^\intercal d\bb F_{\bb\lambda(i)}(\bb t)\hat{\bb\varphi}_i(\bb t)\dd\bb t\dd\bb s\end{aligned}
\end{align}
and the second by
\begin{align}\label{eq:MSLref2}
\begin{aligned}
&\int_\Gamma\int_\Gamma G_k(\bb x-\bb y)\div_\Gamma\bb\varphi_j(\bb x)\div_\Gamma\bb\varphi_i(\bb y)\dd\sigma_{\bb y}\dd\sigma_{\bb x}\\
&{}\qquad\qquad=\int_\square\int_\square G_{\kappa,\bb\lambda(i),\bb\lambda(j)}(\bb s,\bb t)\div\hat{\bb\varphi}_j(\bb s)\div\hat{\bb\varphi}_i(\bb t)\dd\bb t\dd\bb s.\end{aligned}
\end{align}
Assuming that a finite dimensional basis of $\bb\S^*_{p,m}(\Gamma)$ is given
in terms of scalar functions, i.e.,
\[
\bigg\{
\begin{bmatrix}
\varphi_i\\
0
\end{bmatrix},
\begin{bmatrix}
0\\
\varphi_j
\end{bmatrix}
\colon
\varphi_i,\varphi_j~\text{basis
	functions of}~\S^*_{p,m}(\Gamma)
\bigg\},
\]
the matrix $\bb V_{\kappa,h}^*$ can be further decomposed into
\[
\bb V_{\kappa,h}^*
=
\begin{bmatrix}
\bb V_{\kappa,h}^{(1,1)} & 
\bb V_{\kappa,h}^{(1,2)}\\
\bb V_{\kappa,h}^{(2,1)} &
\bb V_{\kappa,h}^{(2,2)}
\end{bmatrix},
\]
with
\begin{align}\label{eq:Vkh}
\begin{aligned}
\Big[\bb V_{\kappa,h}^{(\alpha,\beta)}\Big]_{i,j}
={}&
\int_\square\int_\square G_{\kappa,\bb\lambda(i),\bb\lambda(j)}(\bb s,\bb t)
\Big(
\langle\partial_\alpha\bb F_{\bb\lambda(i)}(\bb s),\partial_\beta\bb F_{\bb\lambda(j)}(\bb t)\rangle\hat{\varphi}_j(\bb s)\hat{\varphi}_i(\bb t)\\
&\qquad\qquad\qquad\qquad\qquad\qquad\qquad-\frac{1}{\kappa^2}\partial_\alpha\hat{\varphi}_j(\bb s)\partial_\beta\hat{\varphi}_i(\bb t)\Big)\dd\bb t\dd\bb s,
\end{aligned}
\end{align}
for $\alpha,\beta=1,2$.

Here, we denote by $\hat{\varphi}_i$ the pullback of the basis function $\varphi_i$ to the reference domain, i.e.,
\[
\hat{\varphi}_i=\varphi_i\circ\bb F_{\bb\lambda(i)}.
\]
This means that $\hat{\varphi}_i$ is effectively an element of $\S _{p,0}^*(\square)$, i.e., it is the tensor product polynomials.
\begin{remark}
	To obtain efficiency in an actual implementation, one may choose
	to simultaneously assemble the $\bb V_{\kappa,h}^{(\alpha,\beta)}$
	and exploit the symmetry $\bb V_{\kappa,h}^{(2,1)}=
	\big(\bb V_{\kappa,h}^{(1,2)}\big)^{\intercal}$ and the
	symmetry of $\bb V_{\kappa,h}^{(1,1)}$ and $\bb V_{\kappa,h}^{(2,2)}$.
	Employing an element-wise integration scheme avoids redundant
	evaluations of kernel function and geometry. This can be maintained
	in the following compression scheme.
\end{remark}

\subsection{Compression of the System Matrix}\label{sec:FMM}
Due to the non-locality of the fundamental solution $G_\kappa$, the system
matrix $\MSL_{k,h}$ given by \eqref{eq:MSLtested} is densely populated. Its storage and
assembly cost are thus prohibitively expensive for higher-dimensional ansatz
and test spaces, and an efficient numerical implementation with compression technique is
needed.
We follow the approach of \cite{Dolz_2016aa,Dolz_2018aa} to compress the matrices
$\bb V_{\kappa,h}^{(\alpha,\beta)}$, $\alpha,\beta=1,2$, in terms of {a specialized}
fast multipole method, which yields a representation of these matrices in terms
of $\mathcal{H}^2$-matrices, see also \cite{Borm_2010aa}. However, the approach is only applicable to matrices
of the kind
\begin{align*}
\big[\bb A\big]_{i,j}
={}&
\int_\Gamma\int_\Gamma G_{\kappa}(\bb x-\bb y)\varphi_j(\bb x)\varphi_i(\bb y)\dd\sigma_{\bb y}\dd\sigma_{\bb x}\\
={}&
\int_\square\int_\square G_{\kappa,\bb\lambda(i),\bb\lambda(j)}(\bb s,\bb t)\hat{\varphi}_j(\bb s)\hat{\varphi}_i(\bb t)\dd\bb t\dd\bb s,
\end{align*}
which does not readily fit the format of the matrices from \eqref{eq:Vkh} due
to the derivatives of the geometry mappings {contained in the} basis functions {and the involved surface divergences}.
In the following, we will, therefore, adapt the construction to the setting
of the electric single layer operator.

For constructing the \(\mathcal{H}^2\)-matrix representation, 
consider the level-wise Cartesian product \(\mathcal{T}\boxtimes
\mathcal{T}:=\big\{\Gamma_{\bb\lambda}\times\Gamma_{\bb\lambda'}
\colon\Gamma_{\bb\lambda},\Gamma_{\bb\lambda'}\in\mathcal{T}, 
|\bb\lambda|=|\bb\lambda'|\big\}\) of the cluster tree $\mathcal{T}$. 
Compressible matrix blocks are then identified by the 
following \emph{admissibility condition}.

\begin{definition}
	The clusters \(\Gamma_{\bb\lambda}\) and \(\Gamma_{\bb\lambda^\prime}\) 
	with \(|\bb\lambda|=|\bb\lambda^\prime|\) are called \emph{admissible} if
	\begin{equation}\label{eq:admissibility}
	\max\big\{\diam(\Gamma_{\bb\lambda}),\diam(\Gamma_{\bb\lambda^\prime})\big\}
	\leq\eta\dist(\Gamma_{\bb\lambda},\Gamma_{\bb\lambda^\prime})
	\end{equation}
	holds for a fixed \(\eta\in (0,1)\). The largest collection of 
	admissible blocks \(\Gamma_{\bb\lambda}\times\Gamma_{\bb\lambda'}
	\in\mathcal{T}\boxtimes\mathcal{T}\) such that 
	\(\Gamma_{\operatorname{dad}(\bb\lambda)}\times
	\Gamma_{\operatorname{dad}(\bb\lambda')}\) is not admissible 
	forms the \emph{far-field} \(\mathcal{F}\subset\mathcal{T}
	\boxtimes\mathcal{T}\)of the operator. The remaining 
	non-admissible blocks correspond to the \emph{near-field} 
	\(\mathcal{N}\subset\mathcal{T}\boxtimes\mathcal{T}\) of the operator.
\end{definition}

The far-field conforms with the compressible matrix blocks, whereas the
near-field is treated by the classical boundary element method, see
Figure~\ref{fig:Hmatrix} for an illustration.
\begin{figure}
\centering
\includegraphics[width=0.3\textwidth]{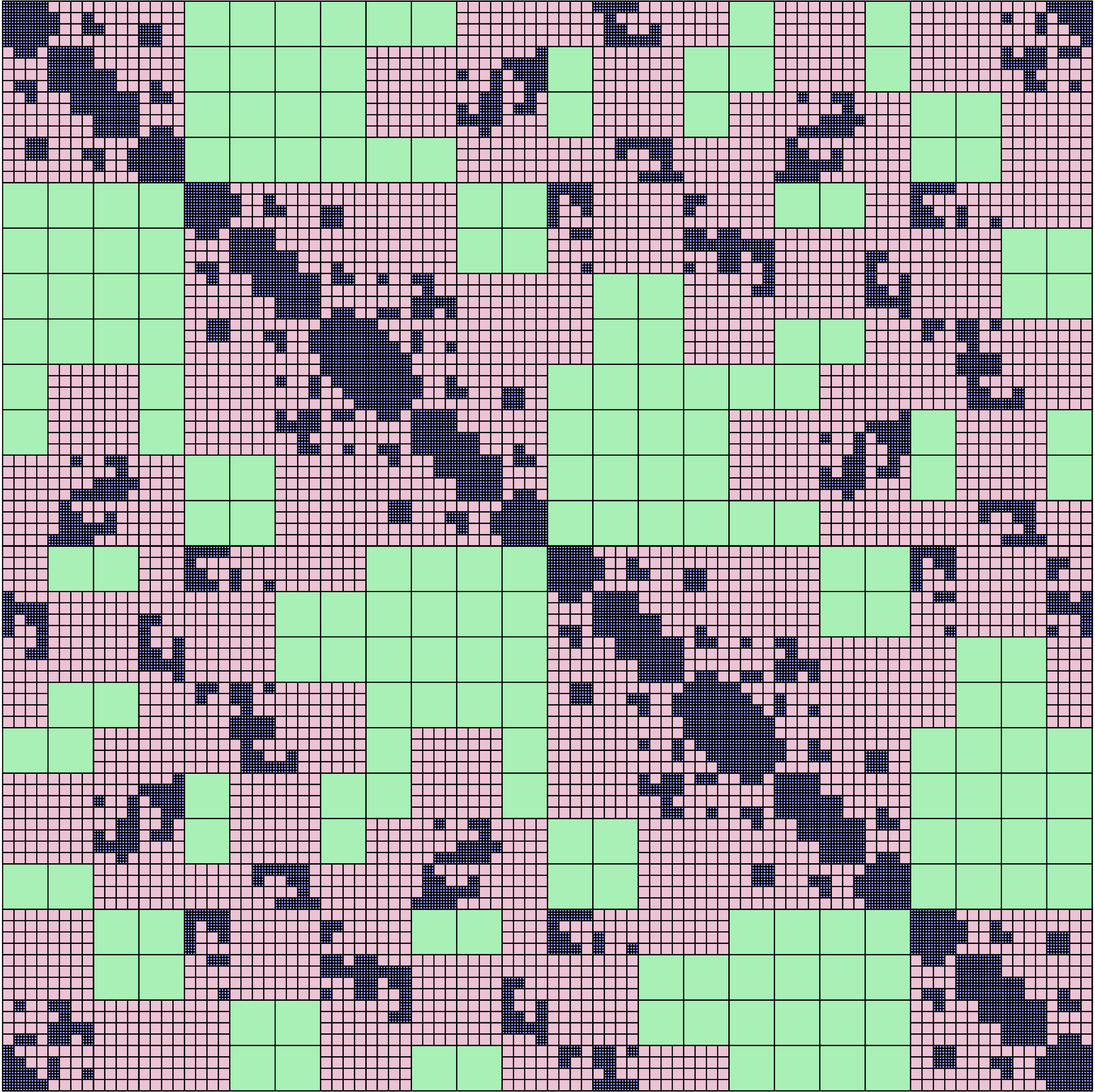}
\caption{\label{fig:Hmatrix}Illustration of the $\mathcal{H}^2$-matrix
partitioning. All but the very smallest blocks are contained in the farfield
and will be compressed by the fast multipole method.}
\end{figure}
The \emph{block-cluster tree} \(\mathcal{B}:=\mathcal{F}
\cup\mathcal{N}\) can be constructed by Algorithm\
\ref{alg:constructblockclustertree}. We remark that for all block-clusters
$\Gamma_{\bb\lambda}\times\Gamma_{\bb\lambda'}\in\mathcal{B}$, it holds
$|\bb\lambda|=|\bb\lambda'|$ and refer to \cite{Dolz_2016aa,Harbrecht_2013ab}
for an in-depth discussion about the special properties of the block-cluster tree
in the isogeometric setting.

\begin{algorithm}[htb]
	\caption{Construction of the block-cluster tree \(\mathcal{B}\)}
	\label{alg:constructblockclustertree}
	\begin{algorithmic}
		\Procedure{BuildBlockClusterTree}{cluster $\Gamma_{\bb\lambda},\Gamma_{\bb\lambda'}$}
		\If {\((\Gamma_{\bb\lambda},\Gamma_{\bb\lambda'})\) is admissible}
		\State $\operatorname{sons}(\Gamma_{\bb\lambda}\times\Gamma_{\bb\lambda'}):=\emptyset$
		\Else
		\State $\operatorname{sons}(\Gamma_{\bb\lambda}\times\Gamma_{\bb\lambda'}):=
		\{\Gamma_{\bb\mu}\times\Gamma_{\bb\mu'}\colon\bb\mu\in\operatorname{sons}(\bb\lambda),\bb\mu'\in\operatorname{sons}(\bb\lambda')\}$
		\For {$\bb\mu\in\operatorname{sons}(\bb\lambda),\bb\mu'\in\operatorname{sons}(\bb\lambda')$}
		\State \Call{BuildBlockClusterTree}{$\Gamma_{\bb\mu}$,$\Gamma_{\bb\mu'}$}
		\EndFor
		\EndIf
		\EndProcedure
	\end{algorithmic}
\end{algorithm}

For a given polynomial degree \(q\in\mathbb{N}\), let
\(\{x_0,x_1,\ldots,x_q\}\subset[0,1]\) denote \(q+1\) interpolation 
points. Furthermore, let \(L_m(s)\) for \(m=0,\ldots,q\) be the 
Lagrangian basis polynomials with respect to these interpolation 
points. By a tensor product construction, one obtains the interpolation 
points \({\bb x}_{\bb m}:=(x_{m_1}, x_{m_2})\) and the corresponding 
tensor product basis polynomials \(L_{\bb m}({\bb s}):= L_{m_1}
(s_1)\cdot L_{m_2}(s_2)\)  for \(m_1,m_2=0,\ldots,q\).
In all admissible 
blocks \(\Gamma_{\bb\lambda}\times\Gamma_{\bb\lambda^\prime}\in\mathcal{F}\), 
this gives rise to the approximation
\[
G_{\kappa,\bb\lambda,\bb\lambda^\prime}({\bb s},{\bb t})\approx
\sum\limits_{\substack{\|{\bb m}\|_\infty\leq q,\\\|{\bb m}^\prime\|_\infty\leq q}}
G_{\kappa,\bb\lambda,\bb\lambda^\prime}({\bb x}_{\bb m},{\bb x}_{{\bb m}^\prime})
L_{\bb m}({\bb s})L_{{\bb m}^\prime}({\bb t})
\eqqcolon
\tilde{G}_{\kappa,\bb\lambda,\bb\lambda^\prime}^{(q)}({\bb s},{\bb t}).
\]
We remark that the approach presented here interpolates the 
localized kernel \eqref{eq:lockernel} via polynomials on the reference domain $\square$
of the isogeometric mappings
rather than the original kernel in space, as first introduced
in \cite{Giebermann_2001aa,Hackbusch_2002aa}. We will see that
this will lead to a complexity of $q^2$ in terms of the interpolation
degree of the compression, rather than $q^3$.

Including the geometry information into the kernel evaluation yields
\begin{align}\label{eq:twomatrices}
\bb V_{\kappa,h}^{(\alpha,\beta)}\Big|_{\bb\lambda,\bb\lambda'}
=
\bb V_{\kappa,h,1}^{(\alpha,\beta)}\Big|_{\bb\lambda,\bb\lambda'}
+
\bb V_{\kappa,h,2}^{(\alpha,\beta)}\Big|_{\bb\lambda,\bb\lambda'}
\end{align}
with
\begin{align*}
&\Big[\bb V_{\kappa,h,1}^{(\alpha,\beta)}\Big|_{\bb\lambda,\bb\lambda'}\Big]_{\ell,\ell'}\\
&{}\qquad=\int_\square\int_\square G_{\kappa,\bb\lambda,\bb\lambda}(\bb s,\bb t)
\langle\partial_\alpha\bb F_{\bb\lambda}(\bb s),\partial_\beta\bb F_{\bb\lambda}(\bb t)\rangle\hat{\varphi}_{\ell'}(\bb s)\hat{\varphi}_\ell(\bb t)\dd\bb t\dd\bb s\\
&{}\qquad\approx
\sum\limits_{\substack{\|{\bb m}\|_\infty\leq q,\\\|{\bb m}^\prime\|_\infty\leq q}}
G_{\kappa,\bb\lambda,\bb\lambda^\prime}({\bb x}_{\bb m},{\bb x}_{{\bb m}^\prime})\langle\partial_\alpha\bb F_{\bb\lambda}(\bb x_{\bb m}),\partial_\beta\bb F_{\bb\lambda}(\bb x_{\bb m'})\rangle\\[-.5cm]
&{}\qquad\qquad\qquad\qquad\qquad\qquad\qquad\cdot
\int_\square
L_{\bb m}({\bb s})\hat{\varphi}_{\ell '}(\bb s)\dd\bb s
\int_\square L_{{\bb m}^\prime}({\bb t}) \hat{\varphi}_\ell(\bb t)\dd\bb t
\end{align*}
for two basis functions \(\hat{\varphi}_\ell,\hat{\varphi}_{\ell^\prime}
\in\S_{p,J-|\bb\lambda|}^*(\square)\). We thus have
the representation
\[
\Big[\bb V_{\kappa,h,1}^{(\alpha,\beta)}\Big|_{\bb\lambda,\bb\lambda'}\Big]_{\ell,\ell'}
=\big[{\bb M}_{|\bb\lambda|}^\square{\bb K}_{\bb\lambda,\bb\lambda^\prime,1}^{(\alpha,\beta)}
({\bb M}_{|\bb\lambda^\prime|}^\square)^\intercal\big]_{\ell,\ell^\prime},
\]
where
\[
\Big[{\bb K}_{\bb\lambda,\bb\lambda^\prime,1}^{(\alpha,\beta)}\Big]_{\bb m,\bb m'}
=
G_{\kappa,\bb\lambda,\bb\lambda^\prime}({\bb x}_{\bb m},{\bb x}_{{\bb m}^\prime})\langle\partial_\alpha\bb F_{\bb\lambda}(\bb x_{\bb m}),\partial_\beta\bb F_{\bb\lambda}(\bb x_{\bb m'})\rangle
\]
and
\begin{align*}
\big[{\bb M}_{|\bb\lambda|}\big] _{m_1,\ell}
={}&{}\int_0^1L_{m_1}(s_1)\hat{\phi}_\ell(s_1)\dd s_1,\quad\hat{\phi}_\ell\in\S_{p,0}^*([0,1]),\\
{\bb M}_{|\bb\lambda|}^{\square}
={}&{}{\bb M}_{|\bb\lambda|}\!\otimes\!{\bb M}_{|\bb\lambda|}.
\end{align*}

For the second term in \eqref{eq:twomatrices} we obtain
\begin{align*}
&\Big[\bb V_{\kappa,h,2}^{(\alpha,\beta)}\Big|_{\bb\lambda,\bb\lambda'}\Big]_{\ell,\ell'}\\
&{}\qquad=-\frac{1}{\kappa ^2}\int_\square\int_\square G_{\kappa,\bb\lambda,\bb\lambda}(\bb s,\bb t)
\partial_{\alpha}\hat{\varphi}_{\ell'}(\bb s)\partial_{\beta}\hat{\varphi}_\ell(\bb t)\dd\bb t\dd\bb s\\
&{}\qquad\approx
\sum\limits_{\substack{\|{\bb m}\|_\infty\leq q,\\\|{\bb m}^\prime\|_\infty\leq q}}
-\frac{1}{\kappa ^2}
G_{\kappa,\bb\lambda,\bb\lambda^\prime}({\bb x}_{\bb m},{\bb x}_{{\bb m}^\prime})
\int_\square
L_{\bb m}({\bb s})\partial_{\alpha}\hat{\varphi}_{\ell '}(\bb s)\dd\bb s
\int_\square L_{{\bb m}^\prime}({\bb t})\partial_{\beta}\hat{\varphi}_\ell(\bb t)\dd\bb t,
\end{align*}
which amounts to the representation
\[
\Big[\bb V_{\kappa,h,2}^{(\alpha,\beta)}\Big|_{\bb\lambda,\bb\lambda'}\Big]_{\ell,\ell'}
=\big[{\bb M}_{|\bb\lambda|}^{\alpha,\square}{\bb K}_{\bb\lambda,\bb\lambda^\prime,2}^{(\alpha,\beta)}
({\bb M}_{|\bb\lambda^\prime|}^{\beta,\square})^\intercal\big]_{\ell,\ell^\prime},
\]
with
\begin{align*}
\big[{\bb M}_{|\bb\lambda|}^{\partial}\big] _{m,\ell}
={}&\int_0^1L_{m}(s)\partial\hat{\phi}_\ell(s)\dd s,\quad\hat{\phi}_\ell\in\S_{p,0}^*([0,1]),\\
{\bb M}_{|\bb\lambda|}^{1,\square}
={}&{\bb M}_{|\bb\lambda|}^\partial\!\otimes\!{\bb M}_{|\bb\lambda|},\\
{\bb M}_{|\bb\lambda|}^{2,\square}
={}&{\bb M}_{|\bb\lambda|}\!\otimes\!{\bb M}_{|\bb\lambda|}^\partial,
\end{align*}
and
\[
\Big[{\bb K}_{\bb\lambda,\bb\lambda^\prime,2}^{(\alpha,\beta)}\Big]_{\bb m,\bb m'}
=
-\frac{1}{\kappa^2}G_{\kappa,\bb\lambda,\bb\lambda^\prime}({\bb x}_{\bb m},{\bb x}_{{\bb m}^\prime}).
\]
In view of \eqref{eq:twomatrices}, this yields the low-rank representation
\begin{align}\label{eq:FMMlr}
\bb V_{\kappa,h}^{(\alpha,\beta)}\Big|_{\bb\lambda,\bb\lambda'}\approx
\begin{bmatrix}
{\bb M}_{|\lambda|}^{\square} & {\bb M}_{|\lambda|}^{\alpha,\square}
\end{bmatrix}
\begin{bmatrix}
{\bb K}_{\bb\lambda,\bb\lambda^\prime,1}^{(\alpha,\beta)} &  \\
& {\bb K}_{\bb\lambda,\bb\lambda^\prime,2}^{(\alpha,\beta)}
\end{bmatrix}
\begin{bmatrix}
\big({\bb M}_{|\lambda|}^{\square}\big)^\intercal \\
\big({\bb M}_{|\lambda|}^{\beta,\square}\big)^\intercal
\end{bmatrix},
\end{align}
for the matrices \eqref{eq:Vkh} in all admissible matrix blocks,
see also Figure~\ref{fig:FMM} for an illustration.

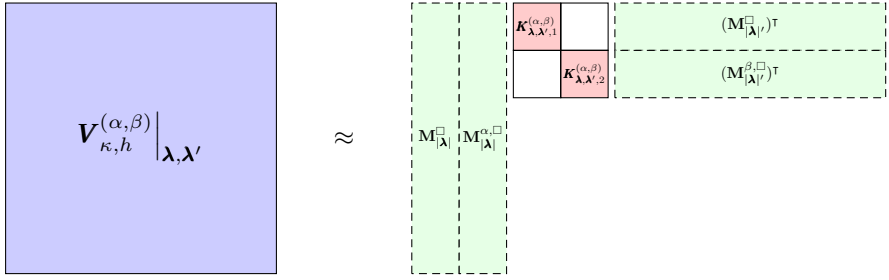
\begin{figure}
\centering
\begin{tikzpicture}[scale=.9]
\draw[fill=blue!20] (-2,-2) rectangle (2,2);
\draw (0, 0) node {$\bb V_{\kappa,h}^{(\alpha,\beta)}\Big|_{\bb\lambda,\bb\lambda'}$};
\draw (3,0) node {$\approx$};
\draw[fill=green!10,densely dashed] (4,-2) rectangle (4.7,2);
\draw (4.35,0) node {\scalebox{0.6}{${\bf M}_{|\bb\lambda|}^\square$}};
\draw[fill=green!10,densely dashed] (4.7,-2) rectangle (5.4,2);
\draw (5.05,0) node {\scalebox{0.6}{${\bf M}_{|\bb\lambda|}^{\alpha,\square}$}};
\draw (5.5,0.6) rectangle (6.2,1.3);
\draw[fill=red!20] (6.2,0.6) rectangle (6.9,1.3);
\draw (6.55,0.95) node {\scalebox{0.5}{${\bb K}_{\bb\lambda,\bb\lambda^\prime,2}^{(\alpha,\beta)}$}};
\draw[fill=red!20] (5.5,1.3) rectangle (6.2,2);
\draw (5.85,1.65) node {\scalebox{0.5}{${\bb K}_{\bb\lambda,\bb\lambda^\prime,1}^{(\alpha,\beta)}$}};
\draw (6.2,1.3) rectangle (6.9,2);
\draw (6.2,0.6) -- (6.2,2);
\draw (5.5,1.3) -- (6.9,1.3);
\draw[fill=green!10,densely dashed] (7,1.3) rectangle (11,2);
\draw (9,1.65) node {\scalebox{0.6}{$({\bf M}_{|\bb\lambda|'}^\square)^\intercal$}};
\draw[fill=green!10,densely dashed] (7,0.6) rectangle (11,1.3);
\draw (9,0.95) node {\scalebox{0.6}{$({\bf M}_{|\bb\lambda|'}^{\beta,\square})^\intercal$}};
\end{tikzpicture}
\caption{\label{fig:FMM}Illustration of the storage savings for
an admissible block
$\bb V_{\kappa,h}^{(\alpha,\beta)}\big|_{\bb\lambda,\bb\lambda'}$ compressed by the fast multipole
method. When using the efficient $\mathcal{H}^2$-variant, ${\bf M}_{|\bb\lambda|}^\square$, 
${\bf M}_{|\bb\lambda|}^{\beta,\square}$, and ${\bf M}_{|\bb\lambda|}^{\beta,\square}$ can
be efficiently represented by recurrence relations such that assembly, storage and application become negligible.}
\end{figure}

We remark that this representation is within the same framework
as it was used for the treatment of the hypersingular operator
for the Laplace equation in \cite{Dolz_2016aa}. Therefore all
considerations made in \cite{Dolz_2016aa} also apply for our setting
here. In particular, there hold the following complexity results,
which amount to a linear scaling w.r.t.~the number of elements.
\begin{theorem}
Let $N$ denote the number of patches and $m$ the level of refinement.
The storage consumption of the compressed matrix has a complexity of
$\mathcal{O}(N\cdot 4^m(pq)^2)$.
Moreover, the matrix-vector multiplication has also a complexity of
$\mathcal{O}(N\cdot 4^m(pq)^2)$, if its fast $\mathcal{H}^2$-variant is
used.
\end{theorem}
\begin{remark}
We stress that the introduced compression scheme has an intrinsic
$\mathcal{H}^2$-structure, which is more efficient than
the frequently used $\mathcal{H}$-matrix structure. Its efficiency
is based on the fact that, for each admissible block
$\bb\lambda\times\bb\lambda'$, there are
only $q^2$ evaluations of the geometry and the kernel function
required to assemble the matrices
${\bb K}_{\bb\lambda,\bb\lambda^\prime,1}^{(\alpha,\beta)}$
and
${\bb K}_{\bb\lambda,\bb\lambda^\prime,2}^{(\alpha,\beta)}$.
The other required matrices from \eqref{eq:FMMlr} can be efficiently
represented by recurrence relations from smaller matrices with
tensor product structure such that assembly, storage and application {do not affect the asymptotic behaviour},
see \cite{Dolz_2016aa}.
\end{remark}

\subsection{Error Analysis Of the Compression Scheme}\label{sec:FMMerror}
The interpolation of the fundamental solution for the compression of the
system matrix introduces an error in the system matrix and, thus, an
error in the numerical solution. Since this error depends on the degree of
the interpolation $q$, this section is dedicated to a suitable
error analysis.
The main application of the following theorem is to bound the approximation error of the
bilinear form in a general form of Strang's first lemma \cite[Thm.~4.2.11]{SS11}.
A direct consequence is that the compression scheme is able to maintain the
convergence rate predicted by Theorem~\ref{thm::quasioptimality}, if the
polynomial degree for the compression is properly chosen.
\begin{theorem}[Error of the Bilinear Form]\label{thm::mutlipole}
	Let $\sigma > 0$ be arbitrary but fixed and denote by $m$ the number
	of uniform refinement steps of $\square$. Then, for the electric
	single layer operator $\MSL_{\kappa,q}$ which results from an interpolation of 
	degree \(q>0\) of the kernel function in every admissible block
	and the exact representation of the kernel in all other blocks, 
	there holds
	\begin{align}\label{eq:bilinearerror}
	\big|\langle \MSL_k\bb u,\bb v\rangle_{\times}-\langle\MSL_{k,q} \bb u, \bb v\rangle_{\times}\big|
	\lesssim 2^{-m\sigma}\|\bb u\|_{\bb H^0(\div_{\Gamma},\Gamma)}\|\bb v\|_{\bb H^0(\div_{\Gamma},\Gamma)},
	\end{align}
	provided that \(q\sim (\sigma+1)m\).
\end{theorem}
\begin{proof}
	The proof is analogous to the proof of \cite[Thm.~5.6]{Harbrecht_2013ab}, applied separately
	to both summands of the electric single layer operator.
\end{proof}
To apply the previous theorem in Strang's first lemma, an additional inverse
estimate of the kind
\[
\|\bb u_h\|_{\bb H^0(\div_{\Gamma},\Gamma)}\lesssim h^{-1/2}\|\bb u_h\|_{\chispace}
\]
on the trial spaces is required. For patchwise continuous spline spaces
$\bb\S _{\bb p,\bb\Xi}^1(\Gamma)$ we provide such an estimate in
Lemma~\ref{lem:inverseestimate}, but we stress that the error analysis is also valid
for other trial spaces providing such an estimate.

We summarize our error analysis in the following theorem, which is a consequence
of the considerations in this section and \cite[Thm.~4.2.11]{SS11}.

\begin{theorem}\label{thm:FMMeverything}
The presented compression scheme maintains the existence and uniqueness of solutions of the numerical scheme. Moreover, {there exists $q_0>0$ such that} the optimal convergence rate of Theorem~\ref{thm::quasioptimality} is maintained if one chooses $q\sim(s+5/2)m$ {and $q\geq q_0$}.
\end{theorem}

%% file: sec_num.tex

A commodity
of fast boundary element methods is that they all rely on iterative solvers
and, thus, they are likely to struggle with high condition numbers caused
by a large {ratio} of wave number to geometry diameter. Thus, for comparison to other
methods, we will indicate both for our test cases. The arising systems are solved via a complex GMRES,
without the application of preconditioners, since a {discussion} of preconditioning would be beyond the scope
of this paper.

The geometry evaluation incorporates Bézier extraction for
efficient geometry evaluations. Matrix assembly, matrix-vector multiplication
and potential evaluation are parallelized via OpenMP \cite{OMPARB_2008aa}.
{The implementation is publicly available under the GNU GPLv3 license \cite{bembel}.}

\subsection{Mie Scattering}\label{numsec::miesphere}
First, we test the implementation via the computation of the surface current induced by a plane wave from a unit sphere. Here, an analytic solution to the density is known in terms of a series expansion, see \cite{Weggler_2011aa} for a comprehensive account. 
Since the energy norm $\norm{\cdot}_\chispace$ of the density is not computable explicitly, we choose to compare the $\bb L^2(\Gamma)$-error of the density.
In accordance to quasi-optimality of the approach, cf.~Theorem \ref{thm::quasioptimality}, a convergence of order $p$ is expected\footnote{We remark again, that with $p$ we refer to the minimal polynomial degree utilized in the construction of the first space of the discrete sequence \eqref{spline::sequence}.}, and can indeed be observed, cf.~Figure \ref{num::sphere::mie}.

\begin{figure}\centering
\centering
	\begin{subfigure}{.48\textwidth}
	\begin{tikzpicture}[scale = .65]
			\begin{axis}[
			height = 7.5cm,
			ymode=log,
			xmode=log,
			x dir = reverse,
			xlabel=reference mesh size $h$,
			xtick={0.5,0.25,0.125,0.06125},
			grid = major,
			xticklabels={$1/2$,$1/4$,$1/8$,$1/16$},
			ylabel=$L^{2}$-error of density,
			legend style={at={(0.0,0.0)},anchor=south west},
			legend columns=2,
			]
			\addplot[line width = 1.5pt,blue,mark = triangle*,mark size=3pt] table [trim cells=true,x=h1,y=L21] {data/plot_sphere_master};
			\addlegendentry{$p=1$}
			\addplot[line width = 1.5pt,red,mark = triangle*,mark size=3pt] table [trim cells=true,x=h1,y=L22] {data/plot_sphere_master};
			\addlegendentry{$p=2$}
			\addplot[line width = 1.5pt,brown,mark = triangle*,mark size=3pt] table [trim cells=true,x=h1,y=L23] {data/plot_sphere_master};
			\addlegendentry{$p=3$}
			\addplot[line width = 1.5pt,gray,mark = triangle*,mark size=3pt] table [trim cells=true,x=h1,y=L24] {data/plot_sphere_master};
			\addlegendentry{$p=4$}
			\addplot[line width = 1.5pt,blue,mark = none,dotted,mark size=3pt] table [trim cells=true,x=h1,y=h1] {data/plot_sphere_master};
			\addlegendentry{$\mathcal O(h^1)$}
			\addplot[line width = 1.5pt,red,mark = none,dotted,mark size=3pt] table [trim cells=true,x=h1,y=h2] {data/plot_sphere_master};
			\addlegendentry{$\mathcal O(h^2)$}
			\addplot[line width = 1.5pt,brown,mark = none,dotted,mark size=3pt] table [trim cells=true,x=h1,y=h3] {data/plot_sphere_master};
			\addlegendentry{$\mathcal O(h^3)$}
			\addplot[line width = 1.5pt,gray,mark = none,dotted,mark size=3pt] table [trim cells=true,x=h1,y=h4] {data/plot_sphere_master};
			\addlegendentry{$\mathcal O(h^4)$}
			\end{axis}
			\end{tikzpicture}
			\subcaption{\footnotesize $L^2(\Gamma)$-error of the density.}
\label{num::sphere::mie}
		\end{subfigure}
		\begin{subfigure}{.48\textwidth}
		\begin{tikzpicture}[scale = .65]
				\begin{axis}[
				height = 7.5cm,
				ymode=log,
				xmode=log,
				ylabel=$\ell^{\infty}$-error of exterior solution,
				grid = major,
				xtick={.5, .25, .125, .0625},
				xlabel=reference mesh size $h$,
				legend style={at={(0.0,0.0)},anchor=south west},
				legend columns=2,
				x dir=reverse,
				xticklabels={$1/2$, $1/4$, $1/8$, $1/16$},
				]
				\addplot[line width = 1.5pt,blue,mark = triangle*,mark size=3pt] table [trim cells=true,x=h1,y=pot1] {data/plot_sphere_master};
				\addlegendentry{$p=1$}
				\addplot[line width = 1.5pt,red,mark = triangle*,mark size=3pt] table [trim cells=true,x=h1,y=pot2] {data/plot_sphere_master};
				\addlegendentry{$p=2$}
				\addplot[line width = 1.5pt,brown,mark = triangle*,mark size=3pt] table [trim cells=true,x=h1,y=pot3] {data/plot_sphere_master};
				\addlegendentry{$p=3$}
				\addplot[line width = 1.5pt,gray,mark = triangle*,mark size=3pt] table [trim cells=true,x=h1,y=pot4] {data/plot_sphere_master};
				\addlegendentry{$p=4$}
				\addplot[line width = 1.5pt,blue,mark = none,dotted,mark size=3pt] table [trim cells=true,x=h1,y=h3pot] {data/plot_sphere_master};
				\addlegendentry{$\mathcal O(h^3)$}
				\addplot[line width = 1.5pt,red,mark = none,dotted,mark size=3pt] table [trim cells=true,x=h1,y=h5pot] {data/plot_sphere_master};
				\addlegendentry{$\mathcal O(h^5)$}
				\addplot[line width = 1.5pt,brown,mark = none,dotted,mark size=3pt] table [trim cells=true,x=h1,y=h7pot] {data/plot_sphere_master};
				\addlegendentry{$\mathcal O(h^7)$}
				\addplot[line width = 1.5pt,gray,mark = none,dotted,mark size=3pt] table [trim cells=true,x=h1,y=h9pot] {data/plot_sphere_master};
				\addlegendentry{$\mathcal O(h^9)$}
				\end{axis}
				\end{tikzpicture}
				\subcaption{\footnotesize $\ell^{\infty}$-error of electric field for $\Dipole_{(0,0.1,0.1)}$.}
\label{num::sphere::man}
				\end{subfigure}
		\caption{Numerical exmples on the unit sphere. Wave number $\kappa = 1$, parameters $q= 10$, and $\eta = 1.6$. The $\Dipole$-error refers to the maximum error obtained via the manufactured solution of a selection of 100 points on a sphere of radius 3 around the origin. GMRES was restarted every 1500 iterations, with a stopping criterion of $\norm{\bb r}_2\leq 10^{-8}$.}
\end{figure}
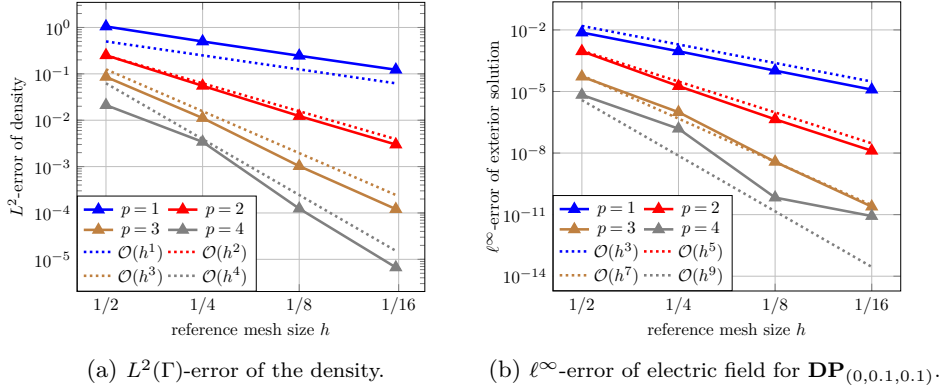

\subsection{The Electric Field as a Quantitiy of Interest}\label{numsec::mieman}

{Although the density obtained in an approach via the electric field integral equation admits a physical interpretation as the surface current, the quantity of interest of scattering problems is mainly the scattered electric field.}

Unfortunately, a numerical implementation of the Mie series {for the computation of the electric field in open space} could not achieve a sufficiently high precision to compare with the high accuracies provided by our isogeometric method. Thus, in order to obtain a reference solution, we employ an approach via manufactured solution, i.e., a function that fulfills the electric wave equation in $\Omega^c$ is used to generate the required Dirichlet data. By existence and uniqueness of the solution, cf.~\cite{Buffa_2003ab}, one can thus validate the numerical scheme.
As such a manufactured solution, we utilize a simple Hertz-Dipole, for which one
can check that it fulfils \eqref{problem::ext_scattering}.
\begin{definition}[Hertz-Dipole, {\cite[{p.~411, (9.18)}]{Jackson_1998aa}}]
Let $\bb x_0\in \Omega^c$. We define the function
\begin{align*}
\Dipole_{\bb x_0}(\bb x)
\coloneqq
e^{i\kappa r}
\bigg(
\frac{\kappa^2}{r}
(\bb n\times\bb p_0)\times\bb n
+
\bigg(
\frac{1}{r^3}-\frac{i\kappa}{r^2}
\bigg)
\big( 3\bb n(\bb n\cdot\bb p_0)-\bb p_0\big)
\bigg),
\end{align*}
with $r=\|\bb x-\bb x_0\|$, $\bb p_0=(0,0.1,0.1)$, and $\bb n=(\bb x-\bb x_0)/r$.
\end{definition}

Given a reference solution, the errors illustrated in Figure \ref{num::sphere::man}
validate the convergence rates of the electric field predicted by
Corollary~\ref{cor:potentialconv}. The last data point of the highest order does not
match the predicted order, but is, with an error around $10^{-12}$, close enough to
machine accuracy to expect noticeable numerical inaccuracies.

Since the sphere example is a classical benchmark test, we choose to publish detailed data about the computation, specifically in terms of time to solution, in Table \ref{tab::sphere}. There, one can also find detailed information about the machine used for computations. This may serve as a reference to compare the presented approach to other implementations, but we stress again that one has to act cautiously when comparing times, since the performance of the fast method depends on various parameters of the problem, in particular, the ratio of the wave number $\kappa$ to the size of the geometry. The input parameters of all computations are detailed in the captions of the corresponding figures. 

Also, we note that due to the efficient, element-based approach of the multipole method, the time spend for matrix assembly is negligible 
compared to the time required for the solution of the linear system, cf.~Table~\ref{tab::sphere}. 

\setlength{\tabcolsep}{11pt}
\begin{table}
	\caption{Detailed data of the unit sphere example with $\kappa = 1$ and $\eta = 1.6$. Computed on a Workstation with Intel(R) Xeon(R) CPU E5-2670 0 @ 2.60GHz, and has been compiled with \texttt{g++ 5.4}, with compile flags \texttt{-O3 -march=native -fopenmp}. Mie-error refers to the error w.r.t.~the analytic solution of the scattering problem described in Section \ref{numsec::miesphere}, while the $\Dipole$-error refers to the error obtained via the manufactured solution as described in Section \ref{numsec::mieman}, cf.~Figures \ref{num::sphere::mie} and \ref{num::sphere::man}. Every 1500 iterations, the GMRES was restarted, with a stopping criterion of $\norm{\bb r}_2\leq 10^{-8}$. Evaluation of the $\Dipole{}$-error was done on a set of points scattered across the sphere of radius 3.}\label{tab::sphere}
	\begin{footnotesize}
	\begin{center}
	\begin{tabular}{|r|llll|}\hline
		&\multicolumn{4}{|c|}{$p = 1$}\\ \hline
		$h$ w.r.t.~$\square$& 0.5 & 0.25 & 0.125 & 0.06125\\
		DOFs {(real, double prec.)}&96&384&1536& 6144 \\
		matrix ass. (s)&0.02&0.14&1.14& 9.15 \\
		solving (s)&0.02&0.32&3.4& 79.9 \\
		GMRES iterations& 12&55&119 & 231 \\
		$\Dipole$-error &0.0074&0.0009&0.0001 & 1.23e-05 \\
		Mie error ($\bb L^2$) &1.051&0.499&0.246& 0.122 \\
		\hline
		& \multicolumn{4}{|c|}{$p = 2$}  \\ \hline
		$h$ w.r.t.~$\square$& 0.5 & 0.25 & 0.125 & 0.06125\\
		DOFs {(real, double prec.)} &216&600&1944	 & 6936 \\
		matrix ass. (s) &0.06&0.55&4.8			 & 47.3 \\
		solving (s) &0.046&2.6&100.6			 & 2279.6 \\
		GMRES iterations&48&158&362				 & 616\\
		$\Dipole$-error &0.0009&1.82e-05&4.41e-07 & 1.29e-08 \\
		Mie error ($\bb L^2$) & 0.251&0.052&0.012			 & 0.0029 \\
		\hline
		&\multicolumn{4}{|c|}{$p = 3$} \\ \hline
		$h$ w.r.t.~$\square$& 0.5 & 0.25 & 0.125 & 0.06125\\
		DOFs {(real, double prec.)}&384&864&2400& 7776 \\
		matrix ass. (s)&0.8&1.16&17.4& 197.3 \\
		solving (s)&0.15&8.46&237.8&8433 \\
		GMRES iterations& 123&294&702 & 2003 \\
		$\Dipole$-error &5.29e-05&9.83e-07&3.72e-09& 2.45e-11 \\
		Mie error ($\bb L^2$) &0.085&0.011&0.0010& 0.000121 \\
		\hline
		&\multicolumn{4}{|c|}{$p = 4$} \\ \hline
		$h$ w.r.t.~$\square$& 0.5 & 0.25 & 0.125 & 0.06125\\
		DOFs {(real, double prec.)}&600&1176&2904& 8664 \\
		matrix ass. (s)&0.6&5.42&52.1& 746.29 \\
		solving (s)&2.08&79.2&3072.9&78508 \\
		GMRES iterations&224&400&919 & 5681\\
		$\Dipole$-error &6.81-e06&1.54e-07&6.77e-11& 8.33e-12 \\
		Mie error ($\bb L^2$) &0.021&0.0034&0.00012& 6.69e-06 \\ \hline
		\end{tabular}
	\end{center}
	\end{footnotesize}
\end{table}

\begin{figure}\centering
	\includegraphics[width=.35\textwidth]{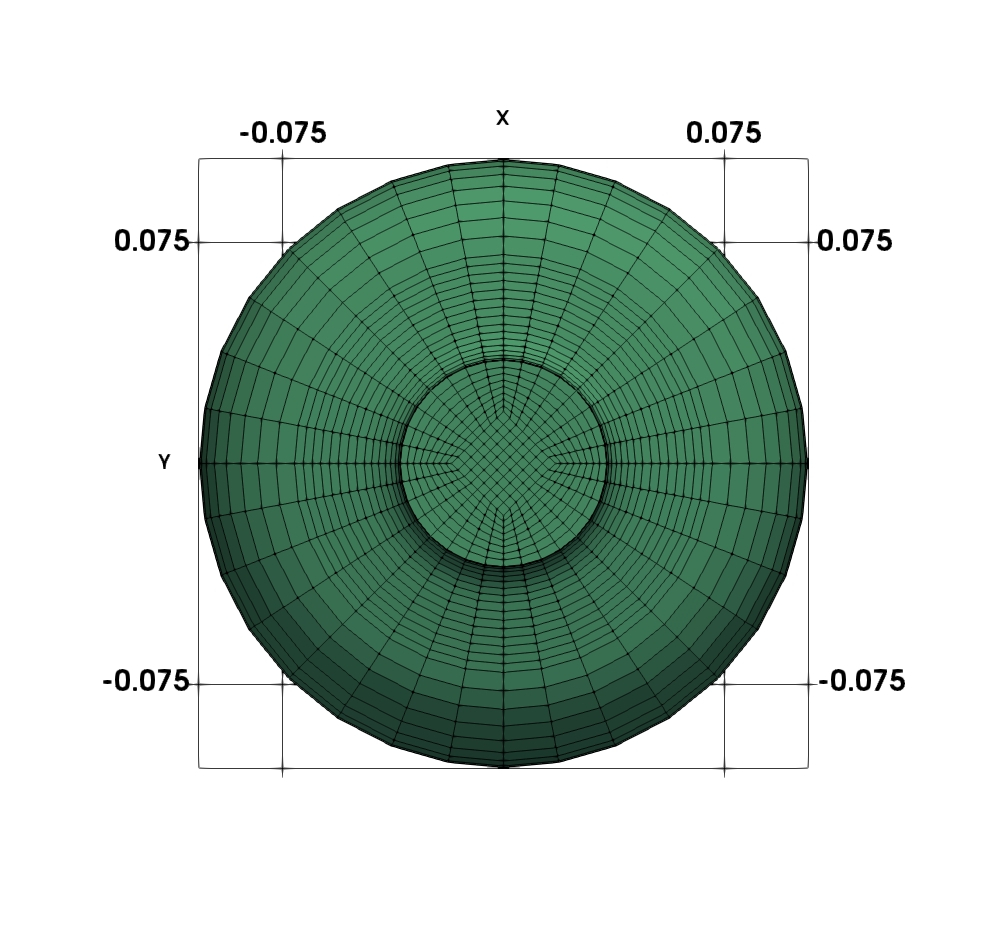}\quad\includegraphics[width=.35\textwidth]{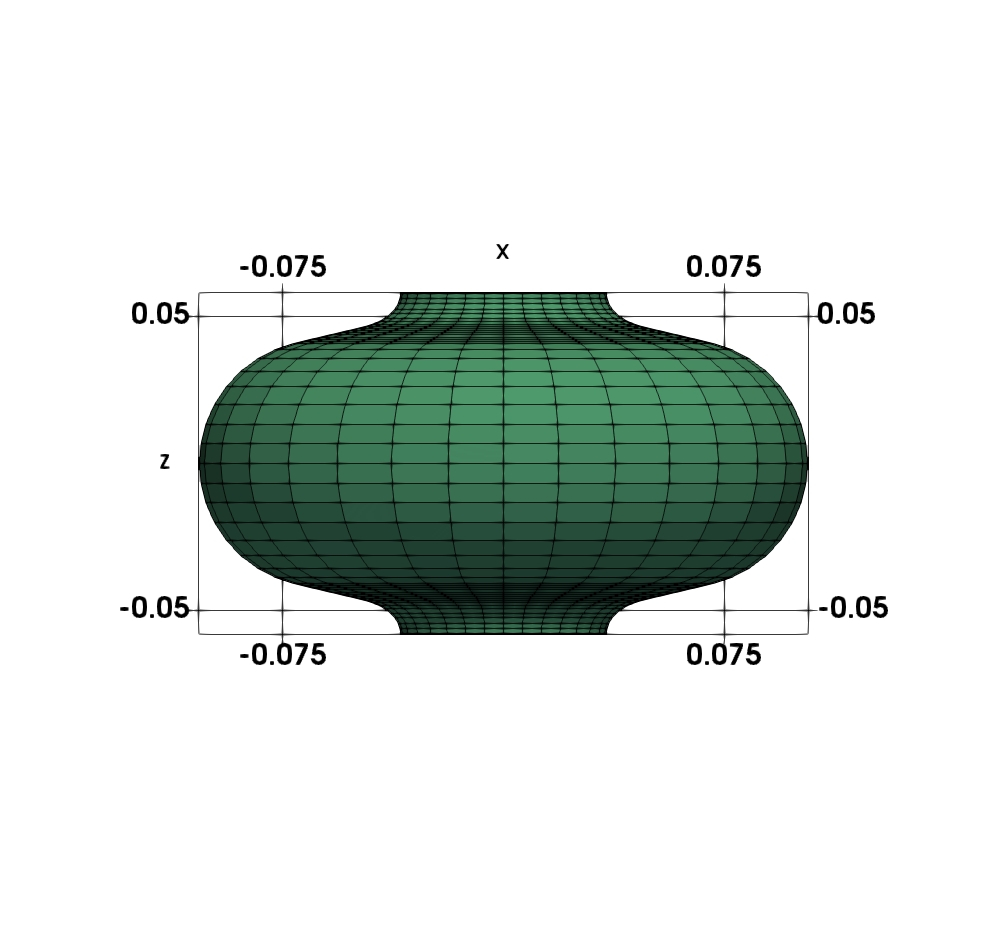}
	\includegraphics[width=.75\textwidth]{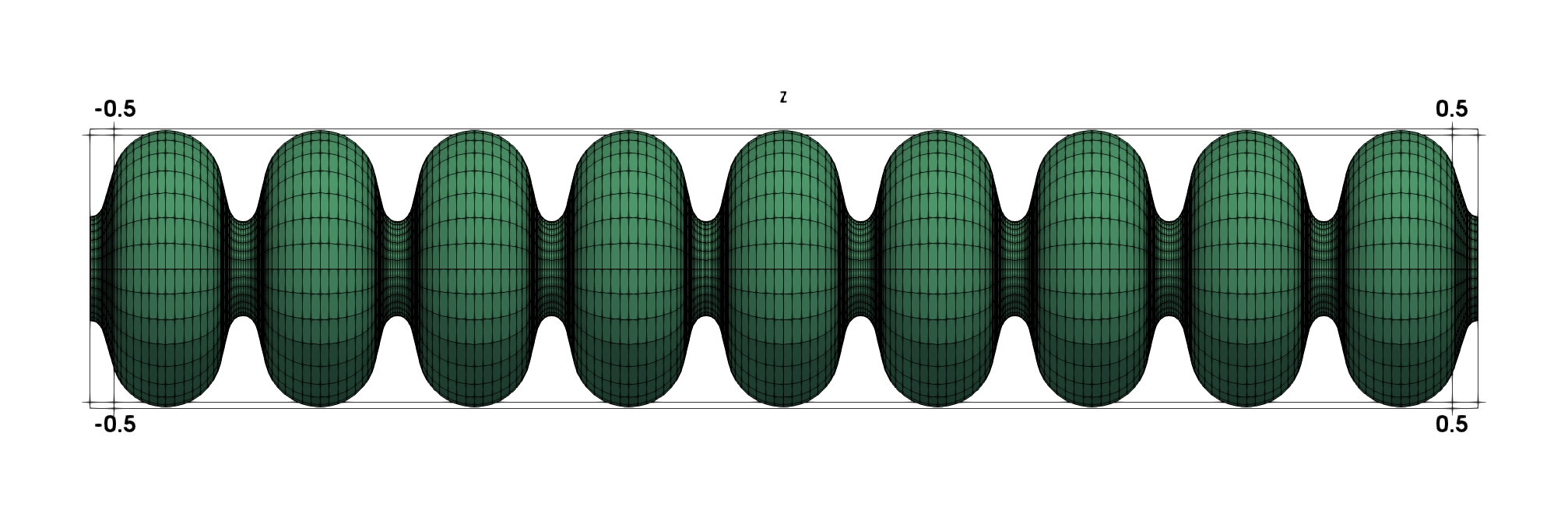}
	\caption{Mesh induced by refinement of level 3 of the Tesla geometries}\label{fig::t1c}
\end{figure}

\subsection{Manufactured Solution: Tesla Cavities}
To test more involved geometries with larger numbers of degrees of freedom, we test our boundary element method on the Tesla cavity geometries. They resemble the cavities as used in particle accelerators, for example at DESY \cite{desy}. Simulation of electromagnetic fields within such cavities is of enormous practical importance, due to the high manufacturing costs through utilization of superconducting materials. 
Thus, one aims for accuracies of the simulation that exceed the tolerances that manufacturers can achieve. Boundary element methods are a good fit for these requirements, due to the high convergence order of pointwise values within the domain, cf.~\cite[Cor.~3.4]{Dolz_2018aa}.

We start these numerical experiments on a single cell of the Tesla cavity, as depicted in Figure \ref{fig::t1c},
which resembles a single cell of the full nine-cell cavity. A volumetric discretization is freely available through the \texttt{geopdes} package of Octave \cite{Falco_2011aa}.
We extracted the boundary in the form of 34 (one-cell) and 226 (nine-cell) quadratic patches of similar sizes, such that all geometry mappings are smooth due to no interior knot repetitions.  On these we apply basis functions of different polynomial degrees, refining uniformly in each refinement step to induce a hierarchical structure cf.~Figure~\ref{fig::refinement}. 
The scattering problem is then solved with a right hand side induced by the Dipole for which the precise parameters are presented in Figure \ref{plot::onecellresults} and Table \ref{tab::tesla}.
The results are depicted in Figure \ref{plot::onecellresults}.
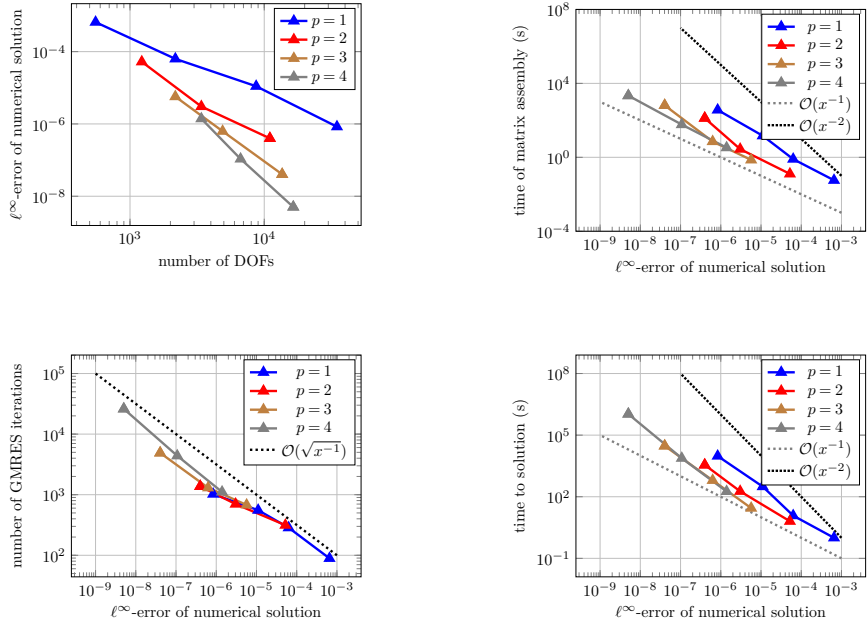
\begin{figure}
\begin{subfigure}{.46\textwidth}
\centering
	\begin{tikzpicture}[scale = .60]
		\begin{axis}[
		height = 6.5cm,
		width = 8cm,
		ymode=log,
		xmode=log,
		xlabel=number of DOFs,
		grid = major,
		ylabel=$\ell^{\infty}$-error of numerical solution,
		]
		\addplot[line width = 1.5pt,blue,mark = triangle*,mark size=3pt] table [trim cells=true,x=dofs1,y=error1] {data/plot_T1C_master};
		\addlegendentry{$p=1$}
		\addplot[line width = 1.5pt,red,mark = triangle*,mark size=3pt] table [trim cells=true,x=dofs2,y=error2] {data/plot_T1C_master};
		\addlegendentry{$p=2$}
		\addplot[line width = 1.5pt,brown,mark = triangle*,mark size=3pt] table [trim cells=true,x=dofs3,y=error3] {data/plot_T1C_master};
		\addlegendentry{$p=3$}
		\addplot[line width = 1.5pt,gray,mark = triangle*,mark size=3pt] table [trim cells=true,x=dofs4,y=error4] {data/plot_T1C_master};
		\addlegendentry{$p=4$}
		\end{axis}
		\end{tikzpicture}
\end{subfigure}\qquad
\begin{subfigure}{.46\textwidth}
\centering
	\begin{tikzpicture}[scale = .60]
		\begin{axis}[
		height = 6.5cm,
		width = 8cm,
		ymode=log,
		xmode=log,
		ylabel=time of matrix assembly (s),
		grid = major,
		xlabel=$\ell^{\infty}$-error of numerical solution,
		]
		\addplot[line width = 1.5pt,blue,mark = triangle*,mark size=3pt] table [trim cells=true,x=error1,y=ass1] {data/plot_T1C_master};
		\addlegendentry{$p=1$}
		\addplot[line width = 1.5pt,red,mark = triangle*,mark size=3pt] table [trim cells=true,x=error2,y=ass2] {data/plot_T1C_master};
		\addlegendentry{$p=2$}
		\addplot[line width = 1.5pt,brown,mark = triangle*,mark size=3pt] table [trim cells=true,x=error3,y=ass3] {data/plot_T1C_master};
		\addlegendentry{$p=3$}
		\addplot[line width = 1.5pt,gray,mark = triangle*,mark size=3pt] table [trim cells=true,x=error4,y=ass4] {data/plot_T1C_master};
		\addlegendentry{$p=4$}
				\addplot[line width = 1.5pt,gray,dotted,mark = none] table [trim cells=true,x=N,y=Nrevs] {data/plot_T1C_master};
		\addlegendentry{$\mathcal{O}(x^{-1})$};
		\addplot[line width = 1.5pt,black,densely dotted,mark = none,mark size=3pt] table [trim cells=true,x=N,y=NNs] {data/plot_T1C_master};
		\addlegendentry{$\mathcal{O}(x^{-2})$};
		\end{axis}
		\end{tikzpicture}
\end{subfigure}\\[1cm]
\begin{subfigure}{.46\textwidth}
\centering
	\begin{tikzpicture}[scale = .60]
		\begin{axis}[
		height = 6.5cm,
		width = 8cm,
		ymode=log,
		xmode=log,
		xlabel=$\ell^{\infty}$-error of numerical solution,
		grid = major,
		ylabel=number of GMRES iterations,
		]
		\addplot[line width = 1.5pt,blue,mark = triangle*,mark size=3pt] table [trim cells=true,x=error1,y=it1] {data/plot_T1C_master};
		\addlegendentry{$p=1$}
		\addplot[line width = 1.5pt,red,mark = triangle*,mark size=3pt] table [trim cells=true,x=error2,y=it2] {data/plot_T1C_master};
		\addlegendentry{$p=2$}
		\addplot[line width = 1.5pt,brown,mark = triangle*,mark size=3pt] table [trim cells=true,x=error3,y=it3] {data/plot_T1C_master};
		\addlegendentry{$p=3$}
		\addplot[line width = 1.5pt,gray,mark = triangle*,mark size=3pt] table [trim cells=true,x=error4,y=it4] {data/plot_T1C_master};
		\addlegendentry{$p=4$}
		\addplot[line width = 1.5pt,black,dotted,mark = none,mark size=3pt] table [trim cells=true,x=N,y=Nsqrt] {data/plot_T1C_master};
		\addlegendentry{$\mathcal O(\sqrt{x^{-1}})$}
		\end{axis}
		\end{tikzpicture}
\end{subfigure}\qquad
\begin{subfigure}{.46\textwidth}
\centering
	\begin{tikzpicture}[scale = .60]
		\begin{axis}[
		height = 6.5cm,
		width = 8cm,
		ymode=log,
		xmode=log,
		ylabel=time to solution (s),
		grid = major,
		xlabel=$\ell^{\infty}$-error of numerical solution,
		]
		\addplot[line width = 1.5pt,blue,mark = triangle*,mark size=3pt] table [trim cells=true,x=error1,y=ts1] {data/plot_T1C_master};
		\addlegendentry{$p=1$}
		\addplot[line width = 1.5pt,red,mark = triangle*,mark size=3pt] table [trim cells=true,x=error2,y=ts2] {data/plot_T1C_master};
		\addlegendentry{$p=2$}
		\addplot[line width = 1.5pt,brown,mark = triangle*,mark size=3pt] table [trim cells=true,x=error3,y=ts3] {data/plot_T1C_master};
		\addlegendentry{$p=3$}
		\addplot[line width = 1.5pt,gray,mark = triangle*,mark size=3pt] table [trim cells=true,x=error4,y=ts4] {data/plot_T1C_master};
		\addlegendentry{$p=4$}
		\addplot[line width = 1.5pt,gray,dotted,mark = none,mark size=3pt] table [trim cells=true,x=N,y=Nrev] {data/plot_T1C_master};
		\addlegendentry{$\mathcal{O}(x^{-1})$};
		\addplot[line width = 1.5pt,black, densely dotted,mark = none,mark size=3pt] table [trim cells=true,x=N,y=NN] {data/plot_T1C_master};
		\addlegendentry{$\mathcal{O}(x^{-2})$};
		\end{axis}
		\end{tikzpicture}
\end{subfigure}
\caption{Results for the Tesla 1-Cell geometry. Wave number $\kappa = 18$, manufacturd solution $\Dipole_{(0,0.1,0.1)}$. Admissibility condition with $\eta=0.1$ and $q=14$. GMRES restart after 1500 iterations, stopping criterion $\norm{\bb r}_2\leq 10^{-10}$. The $\Dipole$-error refers to the maximum error obtained via the manufactured solution of a selection of 100 points on a sphere of radius 3 around the origin. 
}
\label{plot::onecellresults}
\end{figure}

One can still observe the high convergence rates w.r.t.~the number of degrees of freedom. One can also see that the time for matrix assembly, as well as the time to solution, seem to scale  independent of the polynomial degree of the discrete functions.
However, both the time for matrix assembly and the time to solution differ by a constant factor, favouring solutions obtained via higher order approaches.
Moreover, the number of GMRES iterations required for the solution of the system w.r.t.~the achieved accuracy of the solution appears to scale completely independent of $p$. This also favors higher-order approaches: For a set accuracy, systems of higher order approaches are smaller due to the higher accuracy per DOF. Thus, an iteration of a matrix-free solver is computationally cheaper.

For the nine-cell example, such clear behavior is not visible, cf. Table \ref{tab::tesla}. 
We attribute this to the fact that the compression parameters (for admissibility condition and order $q$ of the multipole interpolation) had to be chosen such that the problem remained computable on the accessible machines, i.e., one can not depend on the result of Theorem \ref{thm::mutlipole}.
Despite the suboptimal choice of parameters, one still can observe that the method converges and yields good results.

\begin{table}
	\caption{Detailed data of the sphere example with $\kappa = 10$. 
	Computed on a Workstation, with Intel(R) Xeon(R) CPU E7- 8850, and has been compiled with \texttt{g++ 4.8.5}, with compile flags \texttt{-O3 -march=native -fopenmp}. 
	The $\Dipole$-error refers to the maximum error obtained via the manufactured solution of a selection of 100 points on a sphere of radius 3 around the origin. 
	The stopping criterion for the GMRES was a residual of $\|\bb r\|_2<10^{-10},$ with a restart every 1500 iterations.
	}\label{tab::tesla}
	\begin{footnotesize}
	\begin{center}
	\begin{tabular}{|r|lll|}\hline
		& \multicolumn{3}{|c|}{$p = 2$, $q = 12$, $\eta = 0.15$}  \\ \hline
		$h$ w.r.t.~$\square$& 0.5 & 0.25 & 0.125 \\
		DOFs {(real, double prec.)}& 8136& 22600 &73224\\
		matrix ass. (s)& 43 & 64 & 2031\\
		GMRES iterations& 879 & 1230 & 2552\\
		$\Dipole$-error& 1.69e-03 & 3.84e-07 & 9.79e-09\\
		\hline
		&\multicolumn{3}{|c|}{$p = 3$, $q = 10$, $\eta = 0.3$} \\ \hline
		$h$ w.r.t.~$\square$& 0.5 & 0.25 & 0.125 \\
		DOFs {(real, double prec.)}&14464 & 32544 & 90400\\
		matrix ass. (s)& 37 & 207 & 5944\\
		GMRES iterations& 1424 & 2987 & 7934\\
		$\Dipole$-error& 9.07e-07 & 3.28e-07 & 1.33e-09\\
		\hline
		\end{tabular}
	\end{center}
	\end{footnotesize}
\end{table}

%% file: sec_conclusion.tex

The solution of electromagnetic scattering problems is an important task in
computational engineering. In particular for exterior scattering problems,
the idea of boundary element methods plays well together with the idea of
isogeometric analysis, since boundary representations of geometries are
naturally available from CAD frameworks.

We provided stability assertions for conforming B-spline discretizations
for the electromagnetic scattering problems in the multi-patch case.
Together with recent approximation results \cite{Buffa_2018aa}, this
yields optimal convergence results for the electromagnetic scattering
problem. To cope with the arising dense matrices from the discretization,
we introduced an interpolation-based fast multipole method, which directly
integrates into the isogeometric framework. In particular, and in contrast to
other approaches in the literature, it avoids redundant evaluations of
kernel and geometry and provides simple means for an efficient
$\mathcal{H}^2$-matrix implementation, which provides superior complexity
properties compared to $\mathcal{H}$-matrix implementations. We established
relations between the optimal convergence rates and the compression parameters.
This shows how the compression parameters must be chosen such that the
convergence rates can be maintained.

All theoretical results were confirmed by multiple numerical examples.
We provided detailed data about intensive numerical tests 
and showed that that higher order
approaches yield the extremely high orders of convergence, as predicted 
by the established theory. Moreover, numerical experiments suggest that
approaches via higher order B-splines are favorable over lower order
approaches, w.r.t.~accuracy per DOF and time to solution for the desired
accuracy.

%% file: sec_references.tex

\small

%% file: sec_appendix.tex
For the the conclusion of Theorem~\ref{thm:FMMeverything} we require the following
inverse estimate for $\bb\S_{\bb p,\bb\Xi}^1(\Gamma)$ for discretizations which
feature patch-wise continuous spline spaces. Note that inverse estimates for
$ \S_{\bb p,\bb\Xi}^0(\Gamma)$ and $ \S_{\bb p,\bb\Xi}^2(\Gamma)$ can 
directly be obtained from standard approximation theory for piecewise
polynomials of arbitrary degree.
\begin{lemma}[Inverse Estimate for $\bb\S_{\bb p,\bb\Xi}^1(\Gamma)$]\label{lem:inverseestimate}
Let $\bb\S_{\bb p,\bb\Xi}^1(\Gamma)$ be patchwise continuous. Then it holds
\[
\|\bb v_h\|_{\bb H^0(\div_\Gamma,\Gamma)}\lesssim h^{-1/2}\|\bb v_h\|_{\chispace}\quad\text{for all}~\bb v_h\in\bb\S_{\bb p,\bb\Xi}^1(\Gamma).
\]
\end{lemma}
For the proof we will require some additional knowledge of the space
$\bb H^{-1/2}_\times(\Gamma)$ and an auxiliary lemma. Therefore, we
recall that an equivalent norm on its dual $\bb H_{\times}^{1/2}(\Gamma)$
is given by
\begin{align}\label{eq:H12timesnorm}
\|\bb v_h\|_{\bb H_{\times}^{1/2}(\Gamma)}^2=\sum_{j\leq N}\|\bb v_h\|_{\bb H^{1/2}(\Gamma_j)}^2+
\sum_{j\leq N}\sum_{i\in\mathcal{I}_j}\mathcal{N}_{ij}(\bb v_h),\tag{A}
\end{align}
see also \cite{Buffa_2003ab,Weggler_2014aa}.
For all $j\leq N$ we define $\mathcal{I}_j$ in \eqref{eq:H12timesnorm} as
the set of the indices of the patches sharing an edge with $\Gamma_j$.
Moreover, we define for all $j\leq N$, $i\in\mathcal{I}_j$, the quantity
\[
\mathcal{N}_{ij}(\bb v_h)
=
\int_{\Gamma_i}\int_{\Gamma_j}
\frac{|(\bb v_h\cdot\tilde{\bb n}_{ij})(\bb x)-(\bb v_h\cdot\tilde{\bb n}_{ji})(\bb y)|^2}{\|\bb x-\bb y\|_{\mathbb{R}^3}^3}
\dd\sigma_{\bb y}\dd\sigma_{\bb x},
\]
where, assuming that $\partial_1\bb F_i=\partial_1\bb F_j$ is the derivative of the
parametrization of the common edge $\overline{\Gamma}_i\cap\overline{\Gamma}_j$, we
define $\tilde{\bb n}_{ij}$ and $\tilde{\bb n}_{ji}$ such that
\begin{align*}
\tilde{\bb n}_{ij}\cdot\iota^{-1}_1(\partial_1\bb F_i)={}&\big\|\iota_1^{-1}(\partial_1\bb F_i)\times\iota_1^{-1}(\partial_2\bb F_i)\big\|_{\mathbb{R}^3}^{-1}&&\text{on}~\Gamma_{i},\\
\tilde{\bb n}_{ji}\cdot\iota^{-1}_1(\partial_1\bb F_j)={}&\big\|\iota_1^{-1}(\partial_1\bb F_j)\times\iota_1^{-1}(\partial_2\bb F_j)\big\|_{\mathbb{R}^3}^{-1}&&\text{on}~\Gamma_{j}.
\end{align*}
In short, this means that $\tilde{\bb n}_{ij}$ and $\tilde{\bb n}_{ji}$ are the outward
pointing normals on the common patch boundary, scaled by a constant.

The following auxilliary lemma provides an inverse estimate for
$\mathcal{N}_{ij}(\bb v_h)$.
\begin{lemma}\label{lem:app:help}
Let $\bb\S_{\bb p,\bb\Xi}^1(\Gamma)$ be patchwise continuous. Then it holds
\[
\mathcal{N}_{ij}(\bb v_h)\lesssim h^{-1}\|\bb v_h\|_{\bb L^2(\Gamma_i\cup\Gamma_j)}^2,\quad\text{for all}~\bb v_h\in\bb\S_{\bb p,\bb\Xi}^1(\Gamma).
\]
\end{lemma}
\begin{proof}
For ease of notation we introduce the notation $\hat{\bb f}\coloneqq \iota_1(\bb F_j)(\bb f)$ for every $j\leq N$.
Due to the assumptions on the parametrization, it holds
\begin{align*}
|\bb v_h\cdot\tilde{\bb n}_{ij}|_{H^1(\Gamma_i)}
\sim
|(\bb v_h\cdot\tilde{\bb n}_{ij})\circ\bb F_i|_{H^1(\square)}
\leq
\big|\big(\hat{\bb v}_h\big)_1\big(\hat{\tilde{\bb n}}_{ij}\big)_1\big|_{H^1(\square)}
+
\big|\big(\hat{\bb v}_h\big)_2\big(\hat{\tilde{\bb n}}_{ij}\big)_2\big|_{H^1(\square)}.
\end{align*}
In particular, the derivatives of
$\hat{\tilde{\bb n}}_{ij}$ are bounded, such that we can use
standard inverse estimates for polynomials to estimate
\begin{align*}
|\bb v_h\cdot\tilde{\bb n}_{ij}|_{H^1(\Gamma_i)}
\lesssim{}&
\big|\big(\hat{\bb v}_h\big)_1\big|_{H^1(\square)}
+
\big|\big(\hat{\bb v}_h\big)_2\big|_{H^1(\square)}\\
\lesssim{}&
h^{-1}
\Big(\big\|\big(\hat{\bb v}_h\big)_1\big\|_{L^2(\square)}
+
\big\|\big(\hat{\bb v}_h\big)_2\big\|_{L^2(\square)}\Big)\\
\lesssim{}&
h^{-1}\big\|\hat{\bb v}_h\big\|_{\bb L^2(\square)}\\
\sim{}&
h^{-1}\big\|{\bb v}_h\big\|_{\bb L^2(\Gamma_i)}.
\end{align*}
This yields
\begin{align}\label{eq:app:semiH1est}
\begin{aligned}
|\bb v_h\cdot\tilde{\bb n}_{ij}|_{H^1(\Gamma_i\cup\Gamma_j)}
={}&
\sqrt{
|\bb v_h\cdot\tilde{\bb n}_{ij}|_{H^1(\Gamma_i)}^2
+
|\bb v_h\cdot\tilde{\bb n}_{ij}|_{H^1(\Gamma_j)}^2}\\
\lesssim{}&
h^{-1}\sqrt{\|\bb v_h\|_{\bb L^2(\Gamma_i)}^2
+
\|\bb v_h\|_{\bb L^2(\Gamma_j)}^2}
=
h^{-1}\|\bb v_h\|_{\bb L^2(\Gamma_i\cup\Gamma_j)}.
\end{aligned}\tag{B}
\end{align}
By exploiting the assumptions on the geometry mappings and 
$\bb v_h\cdot\tilde{\bb n}_{ij}\leq\|\bb v_h\|_{\mathbb{R}^3}\|\tilde{\bb
n}_{ij}\|_{\mathbb{R}^3}$, we moreover conclude
\begin{align}\label{eq:app:L2est}
\|\bb v_h\cdot\tilde{\bb n}_{ij}\|_{L^2(\Gamma_i\cup\Gamma_j)}
\lesssim
\|\bb v_h\|_{\bb L^2(\Gamma_i\cup\Gamma_j)}.\tag{C}
\end{align}
Adding \eqref{eq:app:L2est} to \eqref{eq:app:semiH1est} and exploiting
$1\lesssim h^{-1}$ yields
\begin{align}\label{eq:app:H1est}
\|\bb v_h\cdot\tilde{\bb n}_{ij}\|_{H^1(\Gamma_i\cup\Gamma_j)}
\lesssim
h^{-1}\|\bb v_h\|_{\bb L^2(\Gamma_i\cup\Gamma_j)}.\tag{D}
\end{align}
We conclude the proof by remarking that the $\bb L^2(\Gamma_i\cup\Gamma_j)$-orthogonal projection $\bb\Pi _h$ onto
$\S_{\bb p,\bb\Xi}^1(\Gamma_i\cup\Gamma_j)$ acts as an identity on $\S_{\bb p,\bb\Xi}^1(\Gamma_i\cup\Gamma_j)$.
Thus, setting $\bb v_h=\bb\Pi_h\bb v$, $\bb v\in\tilde{\bb H}^1(\div_{\Gamma},\Gamma_i\cup\Gamma_j)$,
interpolation between \eqref{eq:app:L2est} and \eqref{eq:app:H1est} together with
\[
\mathcal{N}_{ij}(\bb v_h)
\leq
|\bb v_h\cdot\tilde{\bb n}_{ij}|_{H^{1/2}(\Gamma_i\cup\Gamma_j)}^2
\leq
\|\bb v_h\cdot\tilde{\bb n}_{ij}\|_{H^{1/2}(\Gamma_i\cup\Gamma_j)}^2
\]
yields the assertion.
\end{proof}
We are now in position to prove the required inverse estimate.
\begin{proof}[Proof of Lemma~\ref{lem:inverseestimate}]
We recall that the involved norms are given by
\begin{align*}
\|\bb v_h\|_{\bb H^0(\div_\Gamma,\Gamma)}^2
={}&
\|\bb v_h\|_{\bb L^2(\Gamma)}^2+\|\div_{\Gamma}\bb v_h\|_{L^2(\Gamma)}^2,\\
\|\bb v_h\|_{\bb H^{-1/2}_{\times}(\div_\Gamma,\Gamma)}^2
={}&
\|\bb v_h\|_{\bb H^{-1/2}_{\times}(\Gamma)}^2+\|\div_{\Gamma}\bb v_h\|_{H^{-1/2}(\Gamma)}^2.
\end{align*}
Since it holds $\div_{\Gamma}\bb v_h\in\S_{\bb p,\bb\Xi}^2(\Gamma)$ for all
$\bb v_h\in\bb\S_{\bb p,\bb\Xi}^1(\Gamma)$, and therefore
\[
\|\div_{\Gamma}\bb v_h\|_{L^2(\Gamma)}\lesssim h^{-1/2}\|\div_{\Gamma}\bb v_h\|_{H^{-1/2}(\Gamma)},
\]
it remains to deal with $\|\bb v_h\|_{\bb L^2(\Gamma)}$, for which we will use
a duality argument.

Therefore, we remember that an equivalent norm on $\bb H^{1/2}_\times(\Gamma)$
is given by \eqref{eq:H12timesnorm}.
Standard inverse estimates for piecewise polynomials in each component yield
\[
\sum_{j\leq N}\|\bb v_h\|_{\bb H^{1/2}(\Gamma_j)}^2
\lesssim
h^{-1}\sum_{j\leq N}\|\bb v_h\|_{\bb L^2(\Gamma_j)}^2,
\]
whereas Lemma~\ref{lem:app:help} allows to estimate
\[
\sum_{j\leq N}\sum_{i\in\mathcal{I}_j}\mathcal{N}_{ij}(\bb v_h)
\lesssim
h^{-1}\sum_{j\leq N}\sum_{i\in\mathcal{I}_j}
\|\bb v_h\|_{\bb L^2(\Gamma_i\cup\Gamma_j)}^2
\lesssim
h^{-1}\|\bb v_h\|_{\bb L^2(\Gamma)}^2.
\]
This yields
\begin{align}\label{eq:app:invhelp}
\|\bb v_h\|_{\bb H^{1/2}_\times(\Gamma)}\lesssim h^{-1/2}\|\bb v_h\|_{\bb L^2(\Gamma)},\quad\text{for all}~\bb v_h\in\bb\S_{\bb p,\bb \Xi}^1(\Gamma).\tag{E}
\end{align}

Using the $\bb L^2(\Gamma)$-orthogonal projection $\bb\Pi _h$ onto $\bb\S_{\bb p,\bb\Xi}^1(\Gamma)$,
we can now use a duality argument to arrive at
\begin{align*}
\|\bb v_h\|_{\bb L^2(\Gamma)}
=&{}
\sup_{\|\bb w\|_{\bb L^2(\Gamma)=1}}\langle\bb v_h,\bb w\rangle_{\bb L^2(\Gamma)}\\
=&{}
\sup_{\|\bb w\|_{\bb L^2(\Gamma)=1}}\big\langle\bb v_h,\bb\Pi_h\bb w\big\rangle_{\times}\\
=&{}
\|\bb v_h\|_{\bb H_{\times}^{-1/2}(\Gamma)}
\sup_{\|\bb w\|_{\bb L^2(\Gamma)=1}}\big\|\bb\Pi _h\bb w\big\|_{\bb H_{\times}^{1/2}(\Gamma)}
\lesssim
h^{-1/2}\|\bb v_h\|_{\bb H_{\times}^{-1/2}(\Gamma)}.
\end{align*}
In the last step we used the previously derived inverse estimate \eqref{eq:app:invhelp} and the stability of the $\bb L^2(\Gamma)$-projection. This yields the assertion.
\end{proof}

%% file: maxwell_siam.bbl
\footnotesize
\begin{thebibliography}{0}

\bibitem{Aimi_18aa}
A.~Aimi, F.~Calabrò, M.~Diligenti, M.L.~Sampoli, G.~Sangalli, A.~Sestini.
\newblock Efficient assembly based on {B}-spline tailored quadrature rules for the {IgA-SGBEM}.
\newblock {\em Comput. Meth. Appl. Mech. Eng.}, 331:(327--342), 2018.

\bibitem{Babuska_1969aa}
I.~Babuška.
\newblock Error-bounds for finite element method.
\newblock {\em J. Numer. Math.}, 16(4):322--333, 1969.

\bibitem{Beb00}
M.~Bebendorf.
\newblock Approximation of boundary element matrices.
\newblock \emph{Numer. Math.}, 86(4):565--589, 2000.

\bibitem{Beer_2017aa}
G.~Beer, V.~Mallardo, E.~Ruocco, B.~Marussig, J.~Zechner, C.~D\"unser, and T.-P.~Fries.
\newblock Isogeometric boundary element analysis with elasto-plastic
  inclusions. part 2: {3-D} problems.
\newblock {\em Comput. Meth. Appl. Mech. Eng.}, 315(Supplement C):418--433,
  2017.

\bibitem{Beirao-da-Veiga_2014aa}
L.~Beirão~da Veiga, A.~Buffa, G.~Sangalli, and R.~V\'azquez.
\newblock Mathematical analysis of variational isogeometric methods.
\newblock {\em Acta. Num.}, 23:157–--287, 2014.

\bibitem{bembel}
J.~D\"olz, H.~Harbrecht, S.~Kurz, M.~Multerer, S.~Sch\"ops, F.~Wolf. Bembel: The Fast Isogeometric Boundary Element C++ Library for Laplace, Helmholtz, and Electric Wave Equation.
\href{www.bembel.eu}{www.bembel.eu}. Technical Report: arXiv:1906.00785.

\bibitem{Bespalov_2010aa}
A.~Bespalov, N.~Heuer, and R.~Hiptmair.
\newblock Convergence of the natural hp-{BEM} for the electric field integral
  equation on polyhedral surfaces.
\newblock {\em SIAM J. Numer. Anal.}, 48:1518--1529, 2010.

\bibitem{Bontinck_2017ag}
Z.~Bontinck, J.~Corno, H.~De~Gersem, S.~Kurz, A.~Pels,
  S.~Sch\"ops, F.~Wolf, C.~de~Falco, J.~D\"olz, R.~V\'azquez, and U.~R\"omer.
\newblock Recent advances of isogeometric analysis in computational
  electromagnetics.
\newblock {\em ICS Newsletter (International Compumag Society)}, 3, 2017.

\bibitem{Borm_2010aa}
S.~B{\"o}rm.
\newblock {\em Efficient numerical methods for non-local operators},
{\em EMS Tracts in Mathematics}, vol. 14.
\newblock European Mathematical Society (EMS), Z\"urich, 2010.

\bibitem{Buffa_2003aa}
A.~Buffa and S.H.~Christiansen.
\newblock The electric field integral equation on {Lipschitz} screens:
  definitions and numerical approximation.
\newblock {\em Numer. Math.}, 94(2):229--267, 2003.

\bibitem{Buffa_2002aa}
A.~Buffa, M.~Costabel, and C.~Schwab.
\newblock Boundary element method for {Maxwell}'s equations on non-smooth
  domains.
\newblock {\em Numer. Math.}, 92:679--710, 2002.

\bibitem{Buffa_2018aa}
A.~Buffa, J.~D\"olz, S.~Kurz, S.~Sch\"ops, R.~V\'azquez, and F.~Wolf.
{\newblock Multipatch approximation of the de {R}ham sequence and its traces in
  isogeometric analysis.
\newblock \emph{Submitted}. Preprint available: arXiv:1806.01062 [math.NA].}

\bibitem{Buffa_2003ab}
A.~Buffa and R.~Hiptmair.
\newblock {Galerkin} boundary element methods for electromagnetic scattering.
\newblock In {\em Topics in computational wave propagation}, 83--124. Springer, 2003.

\bibitem{Buffa_2004aa}
A.~Buffa and R.~Hiptmair.
\newblock A coercive combined field integral equation
for electromagnetic scattering.
\newblock Siam J. Numer. Anal. 42(2):621--640, 2004.

\bibitem{Buffa_2011aa}
A.~Buffa, J.~Rivas, G.~Sangalli, and R.~V\'azquez.
\newblock Isogeometric discrete differential forms in three dimensions.
\newblock {\em SIAM J. Numer. Anal.}, 49(2):818--844, 2011.

\bibitem{Buffa_2014aa}
A.~Buffa and R.~Vázquez.
\newblock Isogeometric analysis for electromagnetic scattering problems.
\emph{International Conference on Numerical Electromagnetic Modeling and Optimization for RF, Microwave, and Terahertz Applications (NEMO)},
1--3,  Pavia, 2014.

\bibitem{Cottrell_2009aa}
J.~Austin Cottrell, T.J.R.~Hughes, and Y.~Bazilevs.
\newblock {\em Isogeometric Analysis: Toward Integration of {CAD} and {FEA}}.
\newblock Wiley, 2009.

\bibitem{Falco_2011aa}
C.~de~Falco, A.~Reali, and R.~V\'azquez.
\newblock {GeoPDEs}: A research tool for isogeometric analysis of {PDE}s.
\newblock {\em Adv. Eng. Software}, 42:1020--1034, 2011.

\bibitem{desy}
Deutsches Elektronen-Synchrotron DESY. TESLA Technology Collaboration: Cavity Database. \href{http://tesla-new.desy.de/cavity\_database}{tesla-new.desy.de/cavity\_database}. Date of access: 13:17, May 4. 2018.

\bibitem{Dolz_2016aa}
J.~D\"olz, H.~Harbrecht, and M.~Peters.
\newblock An interpolation-based fast multipole method for higher-order
  boundary elements on parametric surfaces.
\newblock {\em Int. J. Numer. Meth. Eng.},
  108(13):1705--1728, 2016.

\bibitem{Dolz_2018aa}
J.~D\"olz, H.~Harbrecht, S.~Kurz, S.~Sch\"ops, and F.~Wolf.
\newblock A fast isogeometric {BEM} for the three dimensional {Laplace}- and
  {Helmholtz} problems.
\newblock {\em Comput. Meth. Appl. Mech. Eng.}, 330:83--101, 2018.

\bibitem{Evans_2014aa}
E.J. Evans, M.A. Scott, X. Li and D.C. Thomas. \newblock Hierarchical T-splines:
Analysis-suitability, B\'ezier extraction, and application as an adaptive basis for isogeometric
analysis. \newblock {\em Comput. Meth. Appl. Mech. Eng.} 284:1--20, 2015.

\bibitem{Feischl_2017aa}
M.~Feischl, G.~Gantner, A.~Haberl, D.~Praetorius.
\newblock Optimal convergence for adaptive {IGA} boundary element methods for weakly-singular integral equations.
\newblock {\em Numer. Math.}, 136:147–182, 2017.

\bibitem{Giebermann_2001aa}
K.~Giebermann.
\newblock Multilevel approximation of boundary integral operators.
\newblock {\em Computing}, 67(3):183--207, 2001.

\bibitem{GR86}
V.~Girault and P.~Raviart.
\newblock {\em Finite {E}lement {M}ethods for {N}avier-{S}tokes {E}quations. {T}heory and {A}lgorithms}.
\newblock Springer, Berlin-Heidelberg, 1986.

\bibitem{greengard1987fast}
L.~Greengard and V.~Rokhlin.
\newblock A fast algorithm for particle simulations.
\newblock {\em J. Comput. Phys.}, 73(2):325--348, 1987.

\bibitem{HN89}
W.~Hackbusch, Z.P.~Nowak.
\newblock On the fast matrix multiplication in the boundary
element method by panel clustering.
\newblock \emph{Numer. Math.} 1989; 54(4):463--491.

\bibitem{Hackbusch_2002aa}
W.~Hackbusch and S.~Börm.
\newblock {$\mathcal{H}^2$}-matrix approximation of integral operators by
interpolation.
\newblock {\em Appl. Numer. Math.}, 43(1):129--143, 2002.

\bibitem{Hac15}
W.~Hackbusch.
\newblock {\em Hierarchical {M}atrices: {A}lgorithms and {A}nalysis}.
\newblock Springer, Heidelberg, 2015.

\bibitem{Harbrecht_2001aa}
H.~Harbrecht.
\newblock {\em Wavelet {Galerkin} schemes for the boundary element method in
  three dimensions}.
\newblock PhD thesis, Technische Universit\"at Chemnitz, 2001.

\bibitem{Harbrecht_2013ab}
H.~Harbrecht and M.~Peters.
\newblock Comparison of fast boundary element methods on parametric surfaces.
\newblock {\em Comput. Methods Appl. Mech. Engrg.}, 261--262:39--55, 2013.

\bibitem{Hiptmair_2002aa}
R.~Hiptmair.
\newblock Finite elements in computational electromagnetism.
\newblock {\em Acta. Num.}, 11:237--339, 2002.

\bibitem{Hughes_2005aa}
T.J.R.~Hughes, J.A.~Cottrell, and Y.~Bazilevs.
\newblock Isogeometric analysis: {CAD}, finite elements, {NURBS}, exact
  geometry and mesh refinement.
\newblock {\em Comput. Meth. Appl. Mech. Eng.}, 194:4135--4195, 2005.

\bibitem{Jackson_1998aa}
J.D.~Jackson.
\newblock {\em Classical Electrodynamics}.
\newblock Wiley and Sons, New York, 3rd edition, 1998.

\bibitem{Kurz_2002aa}
S.~Kurz, O.~Rain and S.~Rjasanow.
\newblock The adaptive cross-approximation technique for the {3D} boundary-element method.
\newblock {\em IEEE Trans. Magn.}, 38(2):421--424, 2002.

\bibitem{Lie_2016_aa}
J.~Li and D.~Dault and B.~Liu and Y.~Tong and B.~Shanker.
\newblock Subdivision based isogeometric analysis technique for electric field integral equations for simply connected structures.
\newblock {\em J. Comput. Phys.}, 319:145--162, 2016.

\bibitem{Marussig_2015aa}
B.~Marussig, J.~Zechner, G.~Beer, and T.-P.~Fries.
\newblock Fast isogeometric boundary element method based on independent field
  approximation.
\newblock {\em Comput. Meth. Appl. Mech. Eng.}, 284:458--488, 2015.

\bibitem{McLean_2000aa}
W.~McLean.
\newblock {\em Strongly elliptic systems and boundary integral equations}.
\newblock Cambridge University Press, Cambridge,United Kingdom, 2000.

\bibitem{OMPARB_2008aa}
OpenMP Architecture Review Board. OpenMP application program interface version 3.0, 2008.

\bibitem{Peterson_1995aa}
A.F.~Peterson and K.R.~Aberegg.
\newblock Parametric mapping of vector basis functions for surface integral
  equation formulations.
\newblock {\em Appl. Comput. Electromagn. Soc. J.}, 10:107--115, 1995.

\bibitem{Peterson_2006aa}
A.F.~Peterson.
\newblock Mapped vector basis functions for electromagnetic integral equations.
\newblock {\em Synth. Lec. Comput. Electromagn.}, 1(1):1--124, 2006.

\bibitem{Piegl_1997aa}
L.~Piegl and W.~Tiller.
\newblock {\em The {NURBS} Book}.
\newblock Springer, 2 edition, 1997.

\bibitem{WegglerMatrix} S.~Rjasanow, and L.~Weggler.  
\newblock Matrix valued adaptive cross approximation.
\newblock {\em Math. Meth. Appl. Sci.} 40:2522--2531, 2017. 

\bibitem{SS97}
S.A. Sauter and C.~Schwab.
\newblock Quadrature for $hp$-{G}alerkin {BEM} in $\mathbb{R}^3$.
\newblock \emph{Numerische Mathematik} 78(2):211--258, 1997.

\bibitem{SS11}
S.A. Sauter and C.~Schwab.
\newblock {\em Boundary {E}lement {M}ethods}.
\newblock Springer, Berlin-Heidelberg, 2011.

\bibitem{Schumaker_2007aa}
L.L.~Schumaker.
\newblock {\em Spline functions: Basic theory}.
\newblock Cambridge Mathematical Library. Cambridge University Press,
  Cambridge, United Kingdom, 2007.

\bibitem{Simpson_2018aa}
R.N.~Simpson, Z.~Liu, R.~V\'azquez, and J.A.~Evans.
\newblock An isogeometric boundary element method for electromagnetic
  scattering with compatible B-spline discretizations.
\newblock {\em J. Comput. Phys.}, 362:264--289, 2018.

\bibitem{Simpson_2012aa}
R.N. Simpson, S.P.A. Bordas, J.~Trevelyan, and T.~Rabczuk.
\newblock A two-dimensional isogeometric boundary element method for
  elastostatic analysis.
\newblock {\em Comput. Meth. Appl. Mech. Eng.}, 209–212:87--100, 2012.

\bibitem{Takahashi_2012aa}
T.~Takahashi, T.~Matsumoto.
\newblock An application of fast multipole method to isogeometric boundary element method for {Laplace} equation in two dimensions.
\newblock {\em Eng. Anal. Bound. Elem.} 36(12):1766--1775, 2012.

\bibitem{Weggler_2011aa}
L.~Weggler.
\newblock {\em High Order Boundary Element Methods}.
\newblock Dissertation, Universit\"at des Saarlandes, Saarbr\"ucken, 2011.
  
\bibitem{Weggler_2014aa}
L.~Weggler.
\newblock Generalization of tangential trace spaces of $\bb H(\bcurl,\Omega)$ for curvilinear Lipschitz polyhedral domains $\Omega$.
\newblock {\em Math. Meth. Appl. Sci.}, 37:1847-1852, 2014.

\bibitem{Zaglmayr_2006aa}
S.~Zaglmayr.
\newblock High Order Finite Element Methods
for Electromagnetic Field Computation. \emph{Dissertation}. Johannes Kepler Universität Linz, 2006.

\bibitem{Xu_2002aa}
J.~Xu and L.~Zikatanov.
\newblock Some observations on Babuška and Brezzi theories.
\newblock {\em Numer. Math.}, 94(1):195--202, 2002.

\end{thebibliography}
